\documentclass[aop,MSNbibl,seceqn,citesort,dvips]{arximspdf}
\usepackage{mathbh}

\doi{10.1214/10-AOP621}
\volume{40}
\issue{1}
\pubyear{2012}
\firstpage{103}
\lastpage{129}

\makeatletter
\newcommand{\fracc}[2]{{#1/(#2)}}
\newtheorem{theorem}{Theorem}
\newtheorem{proposition}{Proposition}[section]
\newtheorem{lemma}[proposition]{Lemma}
\newtheorem{corollary}[proposition]{Corollary}
\newproclaim{remark}[theorem]{Remark}

\def\lb{\langle\hspace*{-0.13em}\langle}
\def\rb{\rangle\hspace*{-0.13em}\rangle}

\def\SISGlog{G_{p,\log}}
\def\SISG{p}
\def\mtimes{\diamond}
\newcommand{\VV}{\langle M^K_{1,\cdot}\rangle}
\newcommand{\TM}{\mathbf{A}}
\def\TMstar{A^*}
\newcommand{\TMtrans}{\TM^T}
\newcommand{\SG}{a}
\def\MM{\mathbb{M}}
\renewcommand{\P}{\mathbf{P}}
\def\E{\mathbf{E}}
\newcommand{\Var}{\operatorname{\mathbf{Var}}}
\newcommand{\Cov}{\operatorname{\mathbf{Cov}}}
\newcommand{\To}[1]{\stackrel{#1}{\longrightarrow}}
\newcommand{\Ton}[1]{\To{#1 \rightarrow0}}
\newcommand{\limve}{\Ton{\ve}}

\newcommand{\ARG}{\bolds{\cdot}}
\newcommand{\R}{{\mathbb{R}}}
\newcommand{\N}{{\mathbb{N}}}
\newcommand{\Z}{{\mathbb{Z}}}
\newcommand{\CA}{\mathcal{A}}
\newcommand{\CF}{\mathcal{F}}
\newcommand{\1}{\mathbh{1}}
\newcommand{\ve}{\varepsilon}

\def\CF{\mathcal{F}}
\def\ve{\varepsilon}
\def\CA{\mathcal{A}}
\makeatother

\begin{document}
\begin{frontmatter}

\title{Infinite rate mutually catalytic branching in infinitely many colonies: The longtime behavior\thanksref{T1}}
\runtitle{IMUB: The longtime behavior}

\begin{aug}
\author[A]{\fnms{Achim} \snm{Klenke}\corref{}\ead[label=e1]{math@aklenke.de}}
\and
\author[B]{\fnms{Leonid} \snm{Mytnik}\thanksref{t2}\ead[label=e2]{leonid@ie.technion.ac.il}}
\thankstext{T1}{Supported in part by the German--Israeli Foundation
Grant G-807-227.6/2003.}
\thankstext{t2}{Supported in part by Israel Science Foundation Grant
 1162/06.}
\runauthor{A. Klenke and L. Mytnik}
\affiliation{Universit{\"{a}}t Mainz and Technion Haifa}
\address[A]{Institut f\"{u}r Mathematik\\
Johannes Gutenberg-Universit{\"{a}}t Mainz\\
Staudingerweg 9\\
55099 Mainz\\
Germany\\
\printead{e1}} 
\address[B]{Faculty of Industrial Engineering\\
\quad and Management\\
Technion---Israel Institute of Technology\\
Haifa 32000\\
Israel\\
\printead{e2}}
\end{aug}

\received{\smonth{10} \syear{2009}}
\revised{\smonth{5} \syear{2010}}

%
\begin{abstract}
Consider the infinite rate mutually catalytic branching process (IMUB)
constructed in [Infinite rate mutually catalytic branching in
infinitely many colonies. Construction, characterization and convergence (2008)
Preprint] and [\textit{Ann. Probab.} \textbf{38} (2010) 479--497].
For finite initial conditions, we show that only one type survives
in the long run if the interaction kernel is recurrent. On the other hand,
under a slightly stronger condition than transience, we show that both
types can coexist.
\end{abstract}

%

\pdfkeywords{60J65, 60G17, 60J55, 60K35,
Mutually catalytic branching, Levy noise, stochastic differential equations, Trotter product, coexistence, segregation of types}

\end{frontmatter}

\section{Introduction and main results}
\label{S1}
\subsection{Background and motivation}
\label{S1.1}
As a model for mutually catalytic\break branching,
Dawson and Perkins \cite{DawsonPerkins1998} considered the following
system of coupled stochastic differential equations:
%
%
%
\begin{eqnarray}
\label{E1.1}
Y_{i,t}(k)&=& Y_{i,0}(k)+\int_0^t \sum_{l\in S} \CA
(k,l) Y_{i,s}(l)\,ds \nonumber\hspace*{-35pt}
\\[-8pt]
\\[-8pt]
&&{}+\int_0^t (\gamma Y_{1,s}(k)Y_{2,s}(k) )^{1/2}\,
dW_{i,s}(k)\qquad\mbox{for }t\geq0,\  k\in S,\ i=1,2.
\nonumber\hspace*{-35pt}
\end{eqnarray}
Here $S$ is a countable set that is thought of as the site space and
$\gamma>0$ is a~parameter. (In fact, Dawson and Perkins made the
explicit choice $S=\Z^d$.) The matrix $\CA$ is defined by
%
%
\begin{equation}
\label{E1.2}
\CA(k,l)=\TM(k,l)-\1_{\{k=l\}},
\end{equation}
where $\TM$ is the transpose of a stochastic matrix $\TMtrans$
indexed by $S$ such that $\sup_{k\in S}\sum_{l\in S}\TM(k,l)<\infty
$. Note that $\CA^T$ is the $q$-matrix of the continuous time Markov
chain on $S$ with jump kernel $\TMtrans$.
(In fact, Dawson and Perkins assumed that $\TM$ be symmetric but this
is not substantial.) Finally, $(W_i(k), k\in S, i=1,2)$ is an
independent family of one-dimensional Brownian motions.

This is a spatial model for the evolution of two populations $i=1,2$.
$Y_{i,t}(k)$ is the size of the population of type $i$ at site $k\in S$
at time $t$. The individuals migrate on the site space $S$ according to
the discrete space heat flow induced by $\TMtrans$. Furthermore, at
each given site each type of the population undergoes a random dynamic
that can be interpreted as continuous state Feller's branching with a
branching rate proportional to the local size of the respective other
type. Both types undergo (independently) the same branching dynamics
and influence each other in a symmetric way---hence the name mutually
catalytic branching process.

Dawson and Perkins studied the longtime behavior of this
model with summable initial conditions (and symmetric $\TM$) and
established a dichotomy of coexistence versus noncoexistence of types
depending on transience and recurrence of the Markov chain associated
with $\TM$.
Via the self-duality of the mutually catalytic branching process, its total
mass behavior for summable initial conditions provides information
about the local
behavior if the initial condition is infinite and sufficiently homogeneous.
For $x\in[0,\infty)^2$, let $\underline x$ denote the state in
$([0,\infty)^2)^S$ with $\underline x_i(k)=x_i$ for all $k\in S$, $i=1,2$.
Assume that $Y_0=\underline x$. In \cite{DawsonPerkins1998}, Theorem
1.4, it is shown that $Y_t$ converges in
distribution to some random field $Y_\infty$ as $t\to\infty$.
Furthermore (under some mild regularity assumptions on $\TM^T$), we have
\[
\P_{\underline x}[Y_{1,\infty}(0)Y_{2,\infty}(0)>0]>0
\quad\Longleftrightarrow\quad\TMtrans\mbox{ is transient.}
\]
Hence, in the recurrent case, for constant initial conditions, the two types
segregate locally and form clusters. The assumption that the initial
point is
constant can be weakened to an ergodic random initial condition (see
\cite{CoxKlenkePerkins2000}).

The starting point for this work was the wish to get a quantitative description
of the cluster growth in the recurrent case. We only briefly give the
heuristics. Dawson and Perkins also constructed a version of their
process in
continuous space $\R$ instead of $S$. This
process is defined as the solution of the stochastic partial
differential equation
%
%
\begin{equation}
\label{e1.3}
\frac{dY_{i,t}(r)}{dt}=\Delta Y_{i,t}(r) + \sqrt{\gamma
Y_{1,t}(r)Y_{2,t}(r)} \dot{W}_i(t,r)\qquad\mbox{for }r\in\R,\ i=1,2,\hspace*{-35pt}
\end{equation}
where $\dot{W}_1$ and $\dot{W}_2$ are independent space time white
noises and $\Delta$ is the Laplace operator. As the Laplace operator
generates the semigroup of Brownian motion, and since Brownian motion
on the real line is recurrent, here also the types segregate. Now, due
to Brownian scaling, if we denote by~$Y^\gamma$ the solution of (\ref
{e1.3}) with that given value of $\gamma$, then we obtain
%
%
\begin{equation}
\label{E1.4}
\P_{\underline x} \bigl[ \bigl(Y^\gamma_{T}\bigl(r\sqrt{T}\bigr) \bigr)_{r\in\R
}\in\ARG\bigr]
=\P_{\underline x} [ (Y^{\gamma T}_1(r) )_{r\in\R}\in
\ARG].
\end{equation}
Equation (\ref{E1.4}) shows that clusters of $Y_{1,T}$ grow like
$\sqrt
{T}$ and
that a better understanding of the precise cluster formation can be
obtained by letting $\gamma\to\infty$ for fixed time. The process
$X$ that is the limit (in distribution) of~$Y^\gamma$ as $\gamma\to
\infty$ is called the \textit{infinite rate mutually catalytic
branching process (IMUB)}.

In \cite{KM1}, the infinite rate mutually catalytic branching process
$X$ was constructed for $S$ a singleton. It was shown that $Y^\gamma$
converges to $X$ as $\gamma\to\infty$ and $X$ was characterized in
terms of a martingale problem and in terms of its generator. Since the
two types cannot coexist in the limit $\gamma\to\infty$, the proper
state space for the one-colony IMUB is
\[
E:=[0,\infty)^2\setminus(0,\infty)^2.
\]

In \cite{KM2}, the IMUB process $X$ was constructed for countable $S$
via approximate solutions to a related Poisson noise stochastic partial
differential equation. Here $X$ is a strong Markov process that takes
values in a suitable subspace of $E^S$ that fulfills some growth
condition (Liggett--Spitzer space). More precisely, as shown in \cite
{KM2}, Section 1, for $k\in S$, $i=1,2$, the process
%
%
\begin{equation}
\label{Emar}
\MM_{i,t}(k):=X_{i,t}(k)-X_{i,0}(k)-\int_0^t\CA X_{i,s}(k)\,ds,\qquad
t\geq0,
\end{equation}
is an $L^p$-martingale for every $p\in[1,2)$ (but not for $p=2$) that
could be represented as the stochastic integral with respect to a
Poisson point process.

Furthermore, the IMUB process $X$ was characterized as the solution to
a certain martingale problem and it was shown that $Y^\gamma$
converges to $X$ (in distribution). In order to formulate the
martingale problem, we need the notation
\[
x \mtimes y := -(x_1+x_2)(y_1+y_2) + i(x_1-x_2)(y_1-y_2)\qquad
\mbox{for }
x,y\in E
\]
(with $i=\sqrt{-1}$) and
\[
\lb x,y\rb=\sum_{k\in S}x\mtimes y
\]
for $x,y\in E^S$ such that the sum is well-defined.
Then the $E^S$ valued Markov process $X$ with initial value $X_0=x$ is
characterized by the requirement that
%
%
\begin{equation}
\label{EMP}
e^{\lb X_t,y\rb}-e^{\lb x,y\rb}-\int_0^t\lb\CA X_s, y\rb e^{\lb
X_s,y\rb}\,ds,\qquad t\geq0,
\end{equation}
be a martingale for all suitable
$y\in E^S$.

In \cite{KO}, a construction of $X$ is performed via a Trotter type
approximation scheme, see also \cite{Oeler2008}. Loosely speaking, for
given $\ve>0$, consider a process~$X^\ve$ that solves (\ref{E1.1})
with $\gamma=0$ in each interval $[n\ve,(n+1)\ve)$, $n\in\N_0$. At
the times $n\ve$ the state $X^\ve_{n\ve-}$ is replaced by the limit
(as $t\to\infty$) of a solution~$Y$ of (\ref{E1.1}) with $\CA=0$ and
$\gamma>0$ with initial state $Y_0:=X^\ve_{n\ve-}$. That is, the
value $X^\ve_{n\ve-}(k)$ at each colony $k\in S$ is replaced
(independently) by a point in $E$ chosen randomly according to the exit
distribution of planar Brownian motion in $[0,\infty)^2$ started at
$X^\ve_{n\ve-}(k)$. (See Section~\ref{S2.2} for a more detailed
description.) It was shown in \cite{KO} that $X^\ve$ converges as
$\ve\to0$ to a process that solves the martingale problem (\ref{EMP}).

In this paper, we aim at understanding the longtime behavior of $X$
for summable initial states $x$ [and thus, in (\ref{EMP}) we could take
any bounded $y\in E^S$]. Let $\tau$ be the amount of time that two
independent Markov chains with $q$-matrix $\CA^T$ spend together. We
show that if $\tau$ has
infinite expectation (and $\CA^T$ fulfils some modest regularity
condition), then the types cannot coexist in the long run. On the other
hand, if $\CA$ fulfills a condition that is somewhat stronger than $\E
[\tau]<\infty$, then the types can coexist in the long run. As $X$ is
a process with an infinite variance random dynamics, the meta theorem
that relates stability of the longtime behavior to transience of the
migration dynamics does not apply here. It remains open to check if
there are cases where $\E[\tau]<\infty$ but coexistence of types is
impossible. In particular, it would be interesting to know if this
could happen for certain (transient) random walk kernels $\CA$.

\subsection{Results}

Before we present our result, we give a more detailed description the
longtime behavior for the case of finite $\gamma$ and \textit{finite}
initial conditions as stated in \cite{DawsonPerkins1998}.
While segregation of types for constant initial conditions can be
rephrased as ``two types cannot be present at the same site'' in the
longrun, for finite initial conditions, this does not make sense, as
the population dissipates in space anyway. Here, the notion of (local)
segregation of types is replaced by the notion of (global)
noncoexistence of types.

Let
%
%
\begin{equation}
\label{E1.3}
M^Y_{i,t}:=\langle Y_{i,t}, 1\rangle:=\sum_{k\in S}Y_{i,t}(k),\qquad
t\geq0,
\end{equation}
be the total mass process of type $i=1,2$ and assume that
$M^Y_{1,0},M^Y_{2,0}\in(0,\infty)$. Since $M^Y_{1}$ and $M^Y_{2}$ are
(orthogonal) nonnegative martingales, they converge almost surely and,
in fact, also in $L^1$. Denote by $M^Y_{i,\infty}$ the limit
variables. For the case where $\TM$ is a random walk kernel,
Theorem~1.2 of \cite{DawsonPerkins1998} takes the concise form
%
%
\begin{equation}
\label{E1.6}
\P[M^Y_{1,\infty}M^Y_{2,\infty}>0 ] > 0 \quad
\Longleftrightarrow\quad\TM\mbox{ is transient.}
\end{equation}

In order to formulate the result for the more general case, we have to
be a bit more careful. Let $\SG_t$ be the continuous time kernel; that
is,
%
%
\begin{equation}
\label{E1.7}
\SG_t=\exp(\CA t):=\sum_{n=0}^\infty\frac{t^n \CA^n}{n!}=\sum
_{n=0}^\infty e^{-t}\frac{t^n \TM^n}{n!},
\end{equation}
where $\TM^{n}$ and $\CA^n$ denote the matrix powers. Furthermore,
for $t\geq0$, define the Green kernels
%
%
\begin{equation}
\label{E1.8}
G_t(k,l):=\int_0^t \SG_s(k,l)\,ds\quad\mbox{and}\quad G(k,l):=\int
_0^\infty\SG
_s(k,l)\,ds.
\end{equation}
Finally, let
%
%
\begin{equation}
\label{E1.9}
G_t^{*}:=\sup_{k\in S}G_t(k,k).
\end{equation}

We say that coexistence of types is possible if $\P[M^Y_{1,\infty
}M^Y_{2,\infty}>0]>0$ for all initial states $Y_0$ with
$M^Y_{1,0},M^Y_{2,0}\in(0,\infty)$. We say that coexistence of types
is impossible if $\P[M^Y_{1,\infty}M^Y_{2,\infty}>0]=0$ for all
initial states $Y_0$ with
$M^Y_{1,0},M^Y_{2,0}\in[0,\infty)$.\vspace*{1pt}

Now, Theorem 1.2 of \cite{DawsonPerkins1998} states the following.
%
\setcounter{theorem}{-1}

\begin{theorem}[(Theorem 1.2 of \cite{DawsonPerkins1998})]
Assume that $\TM$ is symmetric.
\begin{longlist}[(ii)]
\item[(i)]
If $ \sup_{k\in S}G(k,k)<\infty$, then coexistence of types is possible.
\item[(ii)] If
%
%
\begin{equation}
\label{E1.10}
\inf_{k\in S}\liminf_{t\to\infty} \frac{G_t(k,k)}{G_t^{*}}>0
\end{equation}
and if $\TM$ is recurrent and irreducible, then coexistence of types
is impossible.
\end{longlist}
\end{theorem}

The theorem describes a dichotomy between a stable behavior
(coexistence of types) and an instable or clustering behavior
(segregation of types) depending on properties of the Green function of
the underlying migration dynamics. A~similar dichotomy along the same
line of transience and recurrence (in the case of migration of random
walk type) was observed before for many interacting models with finite
variance dynamics such as the voter model (see \cite
{CliffordSudbury1973,HolleyLiggett1975}), interacting diffusions on a
compact interval \cite{NotoharaShiga1980,Shiga1980a,CoxGreven1994a},
branching random walk \cite{Kallenberg1977}, the generalized
smoothing and potlatch process~\cite{HolleyLiggett1981}, and the
so-called linear systems~(\cite{Liggett1985}, Chapter~IX). If the
local dynamics has moments only up to order $1+\beta$ for some $\beta
\in(0,1)$, then the the critical line between the two regimes may
shift to the point where higher powers of the Green operator
are finite (see~\cite{DawsonFleischmann1985}). In the model we study
in this paper, all moments less than the second are finite, and thus a
dichotomy could be expected close to the transience/recurrence line but
a little shifted to the transient side.\vadjust{\goodbreak}

In order to prepare for the formulation of our theorem and since we
want to get rid of the symmetry assumption that Dawson and Perkins
imposed on $\TM$, we have to introduce some more notation first.
Recall that $\TMtrans(k,l)=\TM(l,k)$ is the transpose of $\TM$ and
define $\SG_t^T$, $G_t^T$ and so on for $\TMtrans$ similarly as $\SG
_t$, $G_t$ and so on for $\TM$. Furthermore, define the symmetried kernels
%
%
\begin{eqnarray}
\label{E5.2}
\bar\SG_t(k,l) &= &\SG_{t/2}^T\SG_{t/2}(k,l) = \sum_{m\in S}\SG
_{t/2}(m,k)\SG_{t/2}(m,l),\nonumber
\\[-8pt]
\\[-8pt]
\bar\CA&=&\tfrac12 (\CA+\CA^T )\quad\mbox{and}\quad\bar
\TM=\tfrac12
(\TM+\TMtrans).
\nonumber
\end{eqnarray}
Note that $(\bar\SG_t)_{t\geq0}$ is a semigroup (and is generated by
$\bar\CA$) if $\CA\CA^T=\CA^T\CA$. In particular, if $S$ is an
Abelian group and $\TM$ is a random walk kernel, then $(\bar\SG
_t)_{t\geq0}$ is the semigroup of the difference of two (rate $1/2$)
random walks and its one step transition matrix is $\bar\TM$.

Define the expected amount of time $\tau$ two independent Markov
chains with $q$-matrix $\CA^T$ spend together:
%
%
\begin{equation}
\label{E5.4_1}
\bar G(k,l):=\lim_{t\to\infty}\bar G_t(k,l)\qquad\mbox{where
}\bar G_t(k,l):=\int_0^t \bar\SG_s(k,l)\,ds.
\end{equation}
As mentioned above, for many models with migration and local finite
variance random fluctuations, finiteness of $\bar G$ is equivalent to stability.
We will show for the IMUB model here that $\bar G(0,0)=\infty$ (plus
some mild regularity conditions) is sufficient for noncoexistence of types.
In order to formulate the regularity conditions properly, we will also need
%
%
\begin{equation}
\label{E1.9_1}
\bar G_t^*:=\sup_{k\in S}\bar G_t(k,k).
\end{equation}

In order to show coexistence of types, we need more refined quantities.
Define
%
%
\begin{eqnarray}
\label{E5.3}
\SISG_{s}(k,l)&:=& \bigl(\SG_s(\TM+\TMtrans)\SG_s^T \bigr)(k,l)\nonumber
\\[-8pt]
\\[-8pt]
&\hspace*{2.8pt}=&\sum_{m,n\in S}\SG_s(k,m)\bigl(\TM(m,n)+\TM(n,m)\bigr)\SG_s(l,n).
\nonumber
\end{eqnarray}
Let us present a very rough heuristic for the appearance of this object.
If we start with a unit mass of type 1 at $k_1$ and a unit mass of type
2 at $k_2$, then, as we will show, the expected mass of type $i$ at
time $s$ at site $m$ is $\SG_s(m,k_i)$. Recall that the types exclude
each other at any given site. If type $i$ is absent at $m$ at time $s$,
then the infinitesimal impact of type $i$ at site $m$ is governed by
the immigration at rate $\TM(m,l)$ from the other sites $l\in S$; that
is, it is of order $\TM\SG_s(m,k_i)$. Summing over all sites, we see
that the expected total ``activity'' is of order $\SISG_s(k_1,k_2)$.
Since the interaction of types has infinite variance, it is not $\SISG
$ itself that is the crucial quantity, but rather we will see that we
will need a logarithmic correction term.
In order quantify the total amount of interaction in the ``transient''
case, we define
%
%
\begin{equation}
\label{E5.4}
\SISGlog(k,l):=\int_0^\infty\SISG_s(k,l)\bigl (1+|\log(\SISG
_s(k,l))| \bigr)\,ds.
\end{equation}
It is easy to check that
\[
\SISGlog(k,l)\geq\int_0^\infty\SISG_s(k,l)\,ds = \bar G(k,l)-\1
_{\{k=l\}} .
\]
Hence, $\SISGlog$ is infinite if $\bar G$ is infinite.

Define the total mass process
%
%
\begin{equation}
\label{Etmp}
M_{i,t}:=\langle X_{i,t},1\rangle
\end{equation}
and assume that $M_0\in[0,\infty)^2$.
Recall the martingales $\MM_i(k)$ from (\ref{Emar}) and note that
$M_{i,t}=\sum_{k\in S}\MM_{i,t}(k)$ since $\sum_{k\in S}\CA(k,l)=0$
for all $l\in S$. Hence, $M_1$ and $M_2$ are
is a nonnegative martingales and therefore the almost sure limits
\[
Z_i:=\lim_{t\to\infty}M_{i,t}
\]
exist, $i=1,2$.

Recall that $\TMtrans$ is a stochastic matrix and that $\CA(k,l)=\TM
(k,l)-\1_{\{k=l\}}$.

\begin{theorem}
\label{T1}
Let $X$ be the infinite rate mutually catalytic branching process with
kernel $\TM$. Assume that the initial state is given by
\[
X_{1,0}=\1_{\{l_1\}}\quad\mbox{and}\quad X_{2,0}=\1_{\{l_2\}}
\]
for some $l_1,l_2\in S$, $l_1\neq l_2$.

If $l_1$ and $l_2$ are such that $\SISGlog(l_1,l_2)$ is
sufficiently small, then there is coexistence of types in the
longtime limit; that is, $\P[Z_1>0, Z_2>0]>0$.
\end{theorem}

\begin{theorem}
\label{T2}
Assume that $\TM$ fulfills
%
%
\begin{equation}
\label{E1.30}
\sup_{k\in S}\TM(k,k)<1
\end{equation}
and
%
%
\begin{equation}
\label{E1.10.1}
c:=\inf_{k, l\in S}\liminf_{t\to\infty} \frac{\bar G_t(k,l)}{\bar
G_t^*}>0.
\end{equation}
Then coexistence of types is impossible; that is, if $M_{0}\in
[0,\infty)^2$, then
\mbox{$Z_1Z_2=0$} almost surely.
\end{theorem}
\begin{remark}
(i)
Note that in the case where $\TM=\TMtrans$, condition (\ref{E1.10.1})
implies irreducibility of $\TM$. Furthermore, if $\TM=\TMtrans$ is
recurrent and irreducible, due to the strong Markov property, (\ref
{E1.10.1}) is equivalent to (\ref{E1.10}).\vadjust{\goodbreak}

(ii)
Assume that $S$ is an Abelian group and $\TM$ is a random walk kernel.
In this case, (\ref{E1.10.1}) is equivalent to
$\bar\TM$ being irreducible and recurrent. In particular, (\ref
{E1.10.1}) holds if $\TM$ is irreducible and recurrent.

(iii)
For symmetric simple random walk on the $d$-dimensional integer lattice
$\Z^d$, $d\geq3$, it is simple to show (using the CLT), that
$\SISGlog(l_1,l_2)\approx c_d \|l_1-l_2\|^{2-d}$ as $\|l_1-l_2\|\to
\infty$, for some constant $c_d\in(0,\infty)$. Hence, the
assumptions of Theorem~\ref{T1} are fulfilled for simple random walk
with the two populations being sufficiently far apart.
\end{remark}

The proofs in \cite{DawsonPerkins1998} heavily rely on second moment methods.
The main difficulty in the proofs here is the lack of second moments.
(For this reason, presumably the statement of Theorem~\ref{T1} fails
under the weaker assumption that only $\bar G<\infty$.) The strategy
of proof for Theorem~\ref{T1} is therefore to introduce for $K>0$ an auxiliary
process $X^K$ whose jumps in each coordinate are suppressed when
they lead out of the square $[0,K]^2$. This is done in such a way
that the coordinate processes (minus the drifts) become square
integrable orthogonal martingales. For these martingales, we use
the conditions on $\SISGlog$ to estimate the conditional quadratic
variation process.

The proof of Theorem~\ref{T2} also uses the auxiliary process $X^K$
and its conditional quadratic variation process, but the arguments are
more involved.

\subsection{Organization of the paper}

In Section~\ref{S2}, we prove Theorem~\ref{T1}. First, we derive
basic properties of Brownian motion in $[0,K]^2$ stopped upon hitting
the boundary such as hitting distribution, moments and so on.
Then we construct the auxiliary process $X^K$ and conclude Theorem~\ref
{T1} via second moment estimates.

In Section~\ref{S3}, we prove Theorem~\ref{T2}.

\section{\texorpdfstring{Coexistence of types, proof of Theorem~\protect\ref{T1}}{Coexistence of types, proof of Theorem 1}}
\label{S2}
In Section~\ref{S2.1}, we perform some preliminary calculations for the
variance of planar Brownian motion in $[0,K]^2$ stopped upon hitting
the boundary. In
Section~\ref{S2.2}, we construct the auxiliary process $X^K$. In
Section~\ref{S2.3},
we show that the coordinates of~$X^K$ are orthogonal martingales
and we compute their conditional quadratic variation. In the final
Section~\ref{S2.4}, we put the ends together to conclude the
proof of Theorem~\ref{T1}.

\subsection{Brownian motion in a square}
\label{S2.1}
Denote by $Q$ the harmonic measure of planar Brownian motion in
$[0,\infty)^2$. That is, if $B=(B_1,B_2)$ is a Brownian
motion in $\R^2$ started at $x\in[0,\infty)^2$ and
\[
\tau_0:=\inf\{t>0\dvtx  B_t\notin(0,\infty)^2\},
\]
then we define
%
%
\begin{equation}
\label{E2.2}
Q_x:=\P_x[B_{\tau_0} \in\ARG].
\end{equation}
If $x=(u,v)\in(0,\infty)^2$, then the harmonic measure $Q_x$ has a
one-dimensional Lebesgue density on $E$
that can be computed explicitly:
%
%
\begin{equation}
\label{E2.3}
Q_{(u,v)} (d(\bar u,\bar v) )=
\cases{\displaystyle
\frac4\pi
\frac{uv \bar u}{4u^2v^2+ (\bar u^2+ v^2-u^2
)^2}\,d\bar u,
&\quad if $\bar v=0$,\vspace*{3pt}\cr\displaystyle
\frac4\pi
\frac{uv \bar v}
{4u^2v^2+ (\bar v^2+u^2-v^2 )^2}\,d\bar v,
&\quad if $\bar u=0$.
}
\end{equation}
Furthermore, trivially, we have that
$Q_x=\delta_x$ if $x\in E$.

In \cite{KM1}, it is shown that for $x\in(0,\infty)^2$ and $p\in
(0,2)$, we have
%
%
\begin{equation}
\label{E2.01.1}
\int y_i Q_x(dy)=x_i
\end{equation}
and
%
%
\begin{equation}
\label{E2.01.2}
\E_x [\tau_0^{p/2} ]<\infty.
\end{equation}
Applying the Burkholder--Davis--Gundy inequality, this implies that
%
%
\begin{equation}
\label{E2.01.3}
\int y_i^p Q_x(dy)<\infty.
\end{equation}
However, for $p=2$, we have
%
%
\begin{equation}
\label{E2.01.4}
\int y_i^2 Q_x(dy)=\infty.
\end{equation}

Now consider planar Brownian motion $B$ on $[0,K]^2$ and its harmonic
measure $Q^K$ defined by
\[
Q^K_x:=\P_x[B_{\tau_K}\in\ARG],
\]
where
\[
\tau_K:=\inf\{t>0\dvtx  B_t\notin(0,K)^2 \}.
\]
Due to the obvious scaling, we can restrict ourselves mostly to $K=1$.
For simplicity, we define $\tau:=\tau_1$.
\begin{lemma}
\label{L2.1.0}
For all $x\in[0,1]^2$ and $i=1,2$, we have
%
%
\begin{equation}
\label{EL2.1.0.0}
\int y_i Q^1_x(dy) = x_i.
\end{equation}
Furthermore,
%
%
\begin{equation}
\label{EL2.1.0.1}
V(x):=\E_x[\tau]=\int(y_i-x_i)^2 Q^1_x(dy)
\end{equation}
and
%
%
\begin{equation}
\label{EL2.1.0.2}
\int(y_1-x_1)(y_2-x_2) Q^1_x(dy)=0.
\end{equation}
\end{lemma}
\begin{pf}
$(B_{i,t\wedge\tau})_{t\geq0}$ is a bounded martingale and $t\wedge
\tau$ is its quadratic variation.
Hence, (\ref{EL2.1.0.0}) and (\ref{EL2.1.0.1}) are simple
consequences of
the optional stopping theorem for martingales. Similarly, (\ref
{EL2.1.0.2}) follows from the fact that the product $(B_{1,t\wedge\tau
}B_{2,t\wedge\tau})_{t\geq0}$ is a bounded martingale.
\end{pf}
\begin{lemma}
\label{L2.1.1}
For $x=(u,v)\in[0,1]^2$, $V(u,v)$ has the Fourier expansion
%
%
\begin{equation}
\label{EBM.02}
\E_{(u,v)}[\tau]=V(u,v)=
\sum_{m=0}^\infty
\sum_{n=0}^\infty c_{m,n} \sin\bigl((2m+1)\pi u \bigr) \sin
\bigl((2n+1)\pi
v \bigr),\hspace*{-35pt}
\end{equation}
where
%
%
\begin{equation}
\label{EBM.01}
c_{m,n}:=\frac{32}{\pi^4}\frac{\fracc{1}{2m+1}\fracc
{1}{2n+1}}{(2m+1)^2+(2n+1)^2}\qquad\mbox{for all }m,n\in\N_0.
\end{equation}
\end{lemma}
\begin{pf}
It is well known that $V$ is the unique solution of
the Poisson equation with Dirichlet boundary condition
%
%
\begin{eqnarray}
\label{EDirProb}
\tfrac12\Delta g&=&-1\qquad\mbox{in }(0,1)^2,\nonumber
\\[-8pt]
\\[-8pt]
g&=&0\qquad\mbox{at }\partial(0,1)^2.
\nonumber
\end{eqnarray}
Denote by $g(u,v)$ the right-hand side in (\ref{EBM.02}). Clearly, $g=0$
at $\partial(0,1)^2$. Furthermore, for
$u,v\in(0,1)$,
%
%
\begin{eqnarray}
\label{EBM.03}
 \quad \frac12\Delta g(u,v)&=&-
\biggl(\frac{4}{\pi}\sum_{m=0}^\infty
\frac{\sin((2m+1)\pi u )}{2m+1} \biggr)
\biggl(\frac{4}{\pi}\sum_{n=0}^\infty
\frac{\sin((2n+1)\pi v )}{2n+1} \biggr)\nonumber\hspace*{-35pt}
\\[-8pt]
\\[-8pt]
 \quad &=&-1,
\nonumber\hspace*{-35pt}
\end{eqnarray}
where the last equality follows from the fact that each factor is the
Fourier series of the function on
$(0,1)$ that is constant $1$. Hence, we have
$g=V$.
\end{pf}
\begin{corollary}
\label{C2.1.2}
For planar Brownian motion $B$ in $[0,K]^2$, we have
\[
\E_{(u,v)}[\tau_K]=K^2 V (u/K, v/K )
\]
and
\[
\Cov_{(u,v)}[B_{i,\tau_K},B_{j,\tau_K}]=K^2 V (u/K, v/K )
\1_{\{i=j\}}.
\]
\end{corollary}
\begin{pf}
This follows from Brownian scaling.
\end{pf}
\begin{lemma}
\label{L2.1.2b}
For all $u,v\in[0,1]$, we have
%
%
\begin{equation}
\label{EL2.1.2b}
V(u,v)\geq2u(1-u) v(1-v).
\end{equation}
\end{lemma}
\begin{pf}
Let $f(u,v)=2u(1-u) v(1-v)$. Then
\[
\tfrac12\Delta f(u,v)=-2[u(1-u)+v(1-v)]\geq-1.\vadjust{\goodbreak}
\]
Hence, by the maximum principle, $f$ is a sub-solution for the Poisson
problem (\ref{EDirProb}) which shows $f\leq V$.
\end{pf}
\begin{proposition}
\label{P2.1.3}
Let $K>0$ and let $B$ be
Brownian motion in $[0,K]^2$. Then for all $u,v\in[0,K]$, we have
%
%
\begin{eqnarray}
\label{EBM.04}
 \qquad \Cov_{(u,v)}[B_{i,\tau_K}, B_{j,\tau_K}]&=&V(u/K,v/K)K^2 \1_{\{i=j\}
}\nonumber
\\[-4pt]
\\[-12pt]
 \quad &\leq&
8 uv \bigl[1+\log(K)+\bigl(\log(1/u)\wedge\log(1/v)\bigr) \bigr] \1_{\{i=j\}}
\nonumber
\end{eqnarray}
and for all $u,v\in[0,K/2]$, we have
%
%
\begin{equation}
\label{EBM.04bb}
\Cov_{(u,v)}[B_{i,\tau_K}, B_{j,\tau_K}] \geq
\tfrac12 uv \1_{\{i=j\}}.
\end{equation}
\end{proposition}
\begin{pf}
The first equality is due to Brownian scaling. Hence, it is enough to
consider the case $K=1$.
Note that (\ref{EBM.04bb}) is an immediate consequence of Lemma~\ref
{L2.1.2b}. Hence, we concentrate on showing (\ref{EBM.04}).

We have to show that
[with $V$ from (\ref{EBM.02})]
%
%
\begin{equation}
\label{EBM.04b}
V(u,v) \leq 8 uv [1+\log(1/u)].
\end{equation}

By symmetry in $u$ and $v$, this implies (\ref{EBM.04}).

As $V(u,v)$ is bounded by the expected time one-dimensional Brownian
motion started at $v$ needs to hit $\{0,1\}$, we have $V(u,v)\leq
v(1-v)\leq v$
for all $u,v\in[0,1]$. Hence, for $u>1/3$, (\ref{EBM.04b}) holds with
the factor 8 replaced by $3/(1+\log(3))\leq2$.

We may and will now assume that $u\leq1/3$. Let $M\in\N$
be such that
$\frac{1}{2M+3}< u\leq\frac{1}{2M+1}$.
We will show that
%
%
\begin{equation}
\label{EBM.05}
V (u, v )
\leq 8 uv [1+\log(M) ]\qquad\mbox{for }v\in[0,1],
M\in\N.
\end{equation}

We estimate $|\sin((2m+1)\pi u)|\leq\pi(2m+1) u$ for $m\leq M-1$ and
$|\sin((2m+1)\pi u)|\leq1$ for $m\geq M$, as well as $|\sin
((2n+1)\pi v)|\leq\pi(2n+1) v$. Hence, we obtain
\begin{eqnarray*}
V(u,v)
&=& \sum_{n=0}^\infty\sum_{m=0}^{M-1}
c_{m,n}\sin\bigl((2m+1)\pi u \bigr)\sin\bigl((2n+1)\pi
v \bigr)\\
&&{} +
\sum_{n=0}^\infty
\sum_{m=M}^\infty
c_{m,n}\sin\bigl((2m+1)\pi u \bigr)\sin\bigl((2n+1)\pi
v \bigr)\\
&\leq& \frac{32uv}{\pi^2}\sum_{m=0}^{M-1}\sum_{n=0}^\infty
\frac{1}{(2m+1)^2+(2n+1)^2}\\
&&{} + \frac{32v}{\pi^3}\sum_{m=M}^\infty\frac{1}{2m+1}\sum
_{n=0}^\infty\frac{1}{(2m+1)^2+(2n+1)^2}\\
&=:& I_M(u,v) + J_M(u,v).
\end{eqnarray*}
The two summands will be estimated separately. First, note that
\begin{eqnarray*}
I_M(u,v) &\leq& \frac{32uv}{\pi^2} \sum_{m=0}^{M-1}\int
_0^\infty
\frac{1}{(2m+1)^2+t^2}\,dt\\
&=& \frac{16 uv}{\pi} \sum_{m=0}^{M-1}\frac1{2m+1} \leq\frac
{16 uv}{\pi} [1+\log(M) ].
\end{eqnarray*}
Similarly, we get (note that $5M\geq2M+3\geq1/u$ by the assumption on $M$)
\[
J_M(u,v) \leq\frac{16 v}{\pi^2} \sum_{m=M}^{\infty}\frac
1{(2m+1)^2} \leq\frac{4 v}{\pi^2} \frac{1}{M}
\leq\frac{20 v}{\pi^2} \frac{1}{2M+3} \leq\frac{20}{\pi
^2} uv.
\]

Summing up and noting that $16/\pi+20/\pi^2\leq8$, we obtain (\ref{EBM.05}).
\end{pf}

\subsection{Construction of the truncated process}
\label{S2.2}
The aim of this section is to construct a process $X^K$ that approaches
$X$ as $K\to\infty$ and which has finite second moments. The idea is
to suppress the large jumps of $X$ so that the remaining jumps have
second moments. It turns out that if we proceed a bit more subtly, then
we can obtain even that the coordinate processes of $X^K$ are
orthogonal square integrable martingales and that we can control the
conditional quadratic variation process. The rough idea is as follows.
The jumps of $X$ can be interpreted as being driven by the positional
changes of planar Brownian motion at its exit points from $[0,\infty
)^2$. For the process $X^K$, we stop this planar Brownian motion when
it exits $[0,K]^2$.

We could proceed in two ways to construct $X^K$:
\begin{longlist}[(2)]
\item[(1)]
We could imitate the SPDE construction of $X$ (see \cite{KM2}) by
replacing the intensity measure on $E$ of the Poisson point process by
a suitable intensity measure on $[0,K]^2\setminus(0,K)^2$.
\item[(2)]
We could imitate the Trotter type construction of $X$ (see \cite{KO})
by replacing the harmonic measure $Q$ on $[0,\infty)^2$ by the
harmonic measure $Q^K$ on $[0,K]^2$.
\end{longlist}
Here, we follow the latter approach. In \cite{KO}, the following was
done in order to construct $X$: For fixed $\ve>0$, consider the
stochastic process $X^\ve$ with values
in $([0,\infty)^2)^S$ with the following dynamics:
\begin{longlist}[(ii)]
\item[(i)]
Within each time interval $[n\ve,(n+1)\ve)$, $n\in\N_0$, $X^\ve$
is the solution of
\[
dX^\ve_{i,t}(k)=(\CA X^\ve_{i,t})(k)\,dt\qquad\mbox{for }t\in\bigl[n\ve
,(n+1)\ve
\bigr),\  k\in S.
\]
Clearly, the explicit solution is
\[
X^\ve_{i,t}(k)=(\SG_{t-n\ve} X^\ve_{i,n\ve})(k)\qquad\mbox{for
}t\in\bigl[n\ve
,(n+1)\ve\bigr).
\]
\item[(ii)]
At time $n\ve$, $X^\ve$ has a discontinuity. Independently, each
coordinate $X^\ve_{n\ve-}(k)=\SG_{\ve}X^\ve_{(n-1)\ve}(k)$ is
replaced by a random element of $E$ drawn according to the distribution
$Q_{X^\ve_{n\ve-}(k)}$.
\end{longlist}
In order for the solution in Step (i) to be well defined, we have to
impose some growth condition on the initial states (see \cite{KO},
Theorem~1). Since here we are interested in finite initial states, this
growth condition is automatically fulfilled.

In \cite{KO}, it was shown that $X^\ve$ converges as $\ve\to0$ to
$X$ in the Skorohod space of paths $[0,\infty)\to([0,\infty)^2)^S$. Define
\[
\tau_K:=\inf\{t>0\dvtx  \langle X_{1,t}+X_{2,t},1\rangle\geq
K/2 \}.
\]
Clearly, $\tau_K$ is a stopping time and by Doob's inequality,
we get
\[
\P[\tau_K<\infty] \leq2 \frac{\langle
X_{1,0}+X_{2,0},1\rangle}{K}.
\]
Now, assume that $\langle x_{1}+x_{2},1\rangle<K$. We construct
$X^{K,\ve}$ with initial condition~$x$ just as $X$ but with two differences:
\begin{longlist}[(2)]
\item[(1)]
In Step (ii) above, we replace $E$ by $[0,K]^2\setminus(0,K)^2$ and
$Q$ by $Q^K$.
\item[(2)]
If $\langle X^{K,\ve}_{1,n\ve}+X^{K,\ve}_{2,n\ve},1\rangle>K/2$,
then Step (ii) above is omitted.
\end{longlist}
Note that Step (i) preserves the total mass, hence once the total mass
exceeds $K/2$, the process $X^{K,\ve}$ is simply the discrete space
heat flow with kernel $\CA$.
Denote by
\[
\tau^{K,\ve} := \inf\{t>0\dvtx  \langle X^{K,\ve}_{1,t}+X^{K,\ve
}_{2,t},1\rangle\geq
K/2 \}
\]
the time when this first happens. Note that due to the strong Markov
property of planar Brownian motion, we have
\[
Q_{x}
=\int_{[0,K]^2\setminus(0,K)^2}Q^K_x (dy ) Q_y.
\]
Hence, $X^\ve$ and $X^{K,\ve}$ can be coupled to coincide almost
surely until $\tau^{K,\ve}$. Since $X^\ve$ converges, this implies
that also $(X^{K,\ve}_{\tau^{K,\ve}\wedge t})_{t\geq0}$ converges
in the Skorohod space as $\ve\to0$. Since the $\CA$ heat flow
clearly exists, in fact, $X^{K,\ve}$ converges as $\ve\to0$ to some
process $X^K$. Clearly,
\[
M^{K,\ve}_{i,t}(k) := X^{K,\ve}_{i,t}(k)-\int_0^t (\CA
X^{K,\ve}_{i,s} )(k)\,ds,\qquad i=1,2,\  k\in S,
\]
are orthogonal square integrable martingales with conditional quadratic
variation process [see (\ref{EL2.1.0.1}) and (\ref{EBM.04})]
\[
\langle M^{K,\ve}_{i}(k) \rangle_{t}=\sum_{n\dvtx  n\ve\leq
{t\wedge\tau^{K,\ve}}}K^2V (X^{K,\ve}_{1,n\ve-}/K, X^{K,\ve
}_{2,n\ve-}/K ).
\]
\subsubsection*{Upper bound for the conditional quadratic variation}
Let $I$ be the identity matrix. Note\vadjust{\goodbreak} that for each $n$ with $(n-1)\ve
<\tau^{K,\ve}$, we either have $X^{K,\ve}_{1,(n-1)\ve}(k)=0$
[which implies $X^{K,\ve}_{1,n\ve-}(k)= (\SG_\ve-I) X^{K,\ve
}_{1,(n-1)\ve}(k)$] or\break $X^{K,\ve}_{2,(n-1)\ve}(k)=0$ [which implies
$X^{K,\ve}_{2,n\ve-}(k)= (\SG_\ve-I) X^{K,\ve}_{2,(n-1)\ve}(k)$].
Also note that by Proposition~\ref{P2.1.3}, we have $K^2V(u/K,v/K)\leq
u h_K(v)$ and $K^2V(u/K,\break v/K)\leq v h_K(u)$, where
%
%
\begin{equation}
\label{E2.4.4}
h_K(u):=8 u \bigl(1+\log(K/u) \bigr).
\end{equation}
Hence, we get
\begin{eqnarray*}
\langle M^{K,\ve}_{i}(k) \rangle_{t}
&\leq& \ve\sum_{n\dvtx  n\ve\leq{t\wedge\tau^{K,\ve}}}
\ve^{-1}(\SG_\ve-I)
X^{K,\ve}_{1,(n-1)\ve}(k)h_K (X^{K,\ve}_{2,n\ve-}(k) )\\[-2pt]
&&\hphantom{\ve\sum_{n\dvtx  n\ve\leq{t\wedge\tau^{K,\ve}}}}{}
+ \ve^{-1}(\SG_\ve-I) X^{K,\ve}_{2,(n-1)\ve}(k)h_K
(X^{K,\ve}_{1,n\ve-}(k) ).
\end{eqnarray*}
Since bounded $L^2$-martingales converge to bounded $L^2$-martingales,
and since $\ve^{-1}(\SG_\ve-I)(k,l)\limve\TM(k,l)$ for $k\neq l$,
we get that
%
%
\begin{equation}
\label{EdefMK}
M^{K}_{i,t}(k) := X^{K}_{i,t}(k)-\int_0^t (\CA X^{K}_{i,s}
)(k)\,ds, \qquad i=1,2,\  k\in S,
\end{equation}
are orthogonal square integrable martingales with conditional quadratic
variation processes
%
%
\begin{equation}
\label{EoM}
  \langle M^{K}_{i}(k) \rangle_{t}\leq\int_0^{t}\bigl (\TM
X^K_{1,s}(k)h_K (X^K_{2,s}(k) )+\TM X^K_{2,s}(k)h_K
(X^K_{1,s}(k) ) \bigr)\,ds.\hspace*{-35pt}
\end{equation}

\subsubsection*{Lower bound for the conditional quadratic variation}
Define
%
%
\begin{equation}
\label{equt:17_04_1}
\tau^K:=\inf\{t>0\dvtx  \langle X^K_{1,t}+X^K_{2,t},1\rangle\geq
K/2 \}.
\end{equation}
By Doob's inequality,
\[
\P[\tau^K<\infty] \leq2\frac{\langle
X_{1,0}+X_{2,0},1\rangle}{K}.
\]
Furthermore, $X^K$ coincides with $X$ (in distribution) until time
$\tau^K$.
By Proposition~\ref{P2.1.3}, we have
\begin{eqnarray*}
\langle M^{K,\ve}_{i}(k) \rangle_{t}
&\geq&\frac12\sum_{n\dvtx  n\ve\leq{t\wedge\tau^{K,\ve}}}
X^{K,\ve}_{1,n\ve-}(k)X^{K,\ve}_{2,n\ve-}(k)\\
&=& \frac12\sum_{n\dvtx  n\ve\leq{t\wedge\tau^{K,\ve}}}
\SG_\ve X^{K,\ve}_{1,(n-1)\ve}(k)\SG_\ve X^{K,\ve}_{2,(n-1)\ve}(k).
\end{eqnarray*}
Recall that
$\SG_\ve(k,l)=e^{-\ve}\1_{\{k=l\}}+e^{-\ve}\ve\TM(k,l)+\cdots$.
Since $X^{K,\ve}_{1,(n-1)\ve}(k)\times\break X^{K,\ve}_{2,(n-1)\ve}(k)=0$ for
all $n\leq\tau^{K,\ve}$, we have
\begin{eqnarray*}
\langle M^{K,\ve}_{i}(k) \rangle_{t}
&\geq&\frac12\ve
e^{-\ve}
\sum_{n\dvtx  n\ve\leq{t\wedge\tau^{K,\ve}}}
\bigl[\TM X^{K,\ve}_{1,(n-1)\ve}(k)X^{K,\ve}_{2,(n-1)\ve}(k)\\
&&\hphantom{\frac12\ve
e^{-\ve}
\sum_{n\dvtx  n\ve\leq{t\wedge\tau^{K,\ve}}}
\bigl[}
{}
+X^{K,\ve}_{1,(n-1)\ve}(k) \TM X^{K,\ve}_{2,(n-1)\ve}(k) \bigr],
\end{eqnarray*}
where we also used $\TM X^{K,\ve}_{i,(n-1)\ve}(k)X^{K,\ve
}_{3-i,(n-1)\ve}(k)=\CA X^{K,\ve}_{i,(n-1)\ve}(k)X^{K,\ve
}_{3-i,(n-1)\ve}(k)$ for $(n-1)\ve<\tau^{K,\ve}$.

Since $(X^{K,\ve}_i)_{\ve>0}$ is a convergent sequence of bounded
square integrable martingales, also the conditional quadratic variation
processes converge and we infer for $t\geq s\geq0$,
%
%
\begin{eqnarray}
\label{Elowbousq}
&&\langle M^{K}_{i}(k) \rangle_{t}- \langle
M^{K}_{i}(k) \rangle_{s}\nonumber
\\[-8pt]
\\[-8pt]
&& \qquad \geq\frac12\int_{s\wedge\tau^K}^{t\wedge\tau^K} \bigl(\TM
X^K_{1,r}(k)X^K_{2,r}(k)+\TM X^K_{2,r}(k)X^K_{1,r}(k) \bigr)\,dr.
\nonumber
\end{eqnarray}

\subsection{Truncated process and martingales}
\label{S2.3}
In order not to interrupt the flow of the argument later, we start
here with a lemma.
\begin{lemma}
\label{L2.4.1}
Let $Y$ and $Z$ be nonpositively correlated nonnegative random
variables and assume that $h\dvtx [0,\infty)\to[0,\infty)$ is concave
and monotone increasing. Then
$
\E[Yh(Z)]\leq\E[Y] h(\E[Z])
$.
\end{lemma}

\begin{pf}
If $\E[Z]=0$, then we even have equality. Now, assume that $\E[Z]>0$.
By concavity of $h$, there exists a real number $ b\in\R$ such that
for all $ z\geq0$,
\[
h(z)\leq h(\E[Z])+(z-\E[Z]) b.
\]
Since $ h $ is nondecreasing, we have $ b\geq0 $ and thus
\[
\E[Y h(Z)]
\leq \E\bigl[Y \bigl(h(\E[Z])+(Z-\E[Z]) b \bigr) \bigr]
\leq \E[Y] h(\E[Z]).
\]
\upqed
\end{pf}
Let $ l_1,l_2\in S$ and let $X^K$ be the truncated process with
initial state $X^{K}_0=(\1_{\{l_1\}},\1_{\{l_2\}})$.

Writing (\ref{EdefMK}) in the form
\[
X^K_{i,t}(k)=X^K_{i,0}(k)+\sum_{l\in S}\int_0^t \SG_{t-s}(k,l)\,
dM_{i,s}^K(l),
\]
and recalling that the $M^K_{i}(l)$ are orthogonal martingales, we get
that the random variables $X^{K}_{1,s}(k)$ and $X^{K}_{2,s}(l)$ are
uncorrelated for $k,l\in S$. Hence, $\TM X^K_{i,s}(l)$ and
$X^K_{3-i,s}(l)$ are uncorrelated. Note that $x\mapsto h_K(x)$
[defined\break
in (\ref{E2.4.4})] is
concave and monotone increasing for $x\leq K$. Hence, by
Lemma~\ref{L2.4.1}, we get that
%
%
\begin{equation}
\label{E2.4.7}
\E[ \langle M^K_i(l) \rangle_t ]
\leq \sum_{j=1}^2\int_0^t
h_K (\E[X^K_{j,s}(l) ] )
\E[\TM X^K_{3-j,s}(l) ]\,ds.
\end{equation}
Denote by
%
%
\begin{equation}
\label{Edeftmp}
M^K_i=\sum_{k\in S}M^K_i(k)= \langle X^K_i,1 \rangle,\qquad i=1,2,\vadjust{\goodbreak}
\end{equation}
the total mass process of $X^K_i$.
Then (\ref{E2.4.7}) implies
%
%
\begin{eqnarray}
\label{E2.4.8}
\Var[M^K_{i,t} ]
&=&\E[ \langle M^K_i \rangle_t ]\nonumber\\
&=&\sum_{l\in S}\E[ \langle M^K_i(l) \rangle_t ]\\
&\leq&\sum_{j=1}^2\int_0^\infty\sum_{l\in
S}h_K ((\SG_sX_{j,0})(l) ) (\TM\SG_sX_{3-j,0})(l)\,ds.\nonumber
\end{eqnarray}
Using again the concavity of $h_K$ and Jensen's inequality [for the
probability measure $l\mapsto(\TM\SG_s)(l,k)$], we get
%
%
\begin{equation}
\label{E2.4.9}
\Var[M^K_{i,t} ]
\leq\sum_{j=1}^2\int_0^\infty\sum_{k\in S} X_{3-j,0}(k) h_K
\biggl(\sum_{l\in
S}(\TM\SG_s)(l,k) (\SG_sX_{j,0})(l) \biggr)\,ds.\hspace*{-35pt}
\end{equation}

Now, recall that the initial states are $X_{i,0}=\1_{\{l_i\}}$ and that
$\SISG_s= (\SG_s(\TM+\TMtrans)\SG_s^T )$. Hence, we obtain
%
%
\begin{equation}
\label{E2.4.10}\quad
\Var[M^K_{i,t} ]
\leq\sum_{j=1}^2\int_0^\infty h_K (\SISG_s(l_1,l_2) )\,ds
\leq8\log(K) \SISGlog(l_1,l_2).
\end{equation}
Hence, $M^{K}_1$ and $M^{K}_2$ are (orthogonal) $L^2$-bounded
martingales and thus converge almost surely and in $L^2$ to some
random variables $M^K_{1,\infty}$ and $M^K_{2,\infty}$ with
$\E[M^K_i]=1$ and $\Var[M^K_{i,\infty}]\leq8\log(K) \SISGlog(l_1,l_2)$.

\subsection{\texorpdfstring{Proof of Theorem~\protect\ref{T1}}{Proof of Theorem 1}}
\label{S2.4}
Clearly, the product $M^{K}_1\cdot M^{K}_2$ is a uniformly
integrable martingale, hence
%
%
\begin{equation}
\label{E2.5.1}
\E[M^K_{1,\infty} M^K_{2,\infty} ]=\E
[M^{K}_{1,0}M^{K}_{2,0} ]=1.
\end{equation}
Write
\[
\hat M^K_i:=M^K_{i,\infty} \1_{\{\tau_K<\infty\}}.
\]
Since we have $X^K=X$ on $\tau^K=\infty$, in order to show
Theorem~\ref{T1}, it
is enough to show that $\E[M^K_{1,\infty}
M^K_{2,\infty} \1_{\{\tau_K=\infty\}} ]>0$. To this end, we compute
\begin{eqnarray*}
&&\E\bigl[M^K_{1,\infty} M^K_{2,\infty} \1_{\{\tau_K=\infty\}}
\bigr]\\
&& \qquad=
\E[M^K_{1,\infty} M^K_{2,\infty} ]
-\E[M^K_{1,\infty} \hat M^K_2 ] -\E[\hat M^{K}_1
M^K_{2,\infty} ]
+\E[\hat M^{K}_1 \hat M^K_2 ].
\end{eqnarray*}
By (\ref{E2.5.1}), it is thus enough to show that
%
%
\begin{equation}
\label{E2.5.2}
\E[M^K_{1,\infty} \hat M^K_2 ]<\tfrac12
\quad\mbox{and}\quad
\E[\hat M^K_1 M^K_{2,\infty} ]<\tfrac12.
\end{equation}
Although $M^K_{1,\infty}$ and $M^K_{2,\infty}$ are uncorrelated,
this is not true for $\hat M^K_1$ and $M^K_{2,\infty}$ (at least
we cannot show this). Hence,\vadjust{\goodbreak} we have to use a slightly more
subtle
argument. We employ the Cauchy--Schwarz inequality to estimate
%
%
\begin{eqnarray}
\label{E2.5.3}
\E[M^K_{1,\infty} \hat M^K_2 ]^2
&\leq&\E[(M^K_{1,\infty})^2 ] \E[(\hat M^K_1)^2
]\nonumber
\\
&=& (1+\Var[M^K_{1,\infty} ] ) \E[(\hat
M^K_1)^2 ]\\
&\leq& \bigl(1+8\log(K) \SISGlog(l_1,l_2) \bigr)
\tfrac{9}{4} K^2.\nonumber
\end{eqnarray}
Recall that $K>2$ is fixed. Now, choose $l_1$ and $l_2$ such that
$\SISGlog(l_1,l_2)$ gets so small that the right-hand side
in
(\ref{E2.5.3}) is bounded by $\frac14$. Similarly, we get $\E[\hat M^K_1
M^K_{2,\infty}]^2<\frac14$.

This shows (\ref{E2.5.2}) and thus completes the proof of
Theorem~\ref{T1}.

\section{\texorpdfstring{Noncoexistence of types, proof of Theorem~\protect\ref{T2}}{Noncoexistence of types, proof of Theorem 2}}
\label{S3}

The strategy of proof is described in the following two steps.

\textit{Step 1.} Replace the process $X$ by the approximate
process $X^K$ constructed in Section~\ref{S2.2}. If $X$ would have
coexistence of types, then so would~$X^K$ (for some large $K$).

\textit{Step 2.}
Since $M^K_{i,t}=\langle X^K_{i,t},1\rangle$, $t\geq0$, is a
convergent martingale with bound\-ed jump size, also the conditional
quadratic variation process $\VV$ converges. We derive a lower bound
for $\VV$ and show that due to the recurrence of $\CA$, this lower
bound would diverge with positive probability if~$X^K$ had coexistence
of types. Hence, there can be no coexistence of types for~$X^K$ and thus neither
for $X$.

\subsection{Step 1: The approximate process}
\label{S3.1}
Assume that for $X$, coexistence of types is possible. We will lead
this assumption to a contradiction.

Recall that $M_{i,t}:=\langle X_{i,t},1\rangle$, $t\geq0$, $i=1,2$,
are the total mass processes.
Coexistence of types means that there exists a deterministic initial
state $X_0$ such that $M_{i,t}<\infty$, $i=1,2$, and such that
\[
\lim_{t\to\infty}M_{1,t} M_{2,t}>0
\qquad\mbox{with positive probability.}
\]
(Recall from the discussion prior to Theorem~\ref{T1} that the total
mass processes are nonnegative martingales and are hence convergent.)
We use Lemma~\ref{Lnew} below (with $Z_t=M_{1,t}M_{2,t}$) to infer
that there exists a $ \delta>0$, such that
%
%
\begin{equation}
\label{Estar}
\P[M_{1,t} M_{2,t}\geq\delta\mbox{ for all }t\geq0 ] \geq
6\delta.
\end{equation}

\begin{lemma}
\label{Lnew}
Let $(Z_t)_{t\geq0}$ be a nonnegative right continuous
supermartingale. Then
\[
\inf\{Z_t\dvtx  t\geq0\}>0 \qquad\mbox{a.s. on the event } \Bigl\{\lim_{t\to
\infty}Z_t>0 \Bigr\}.
\]
\end{lemma}
\begin{pf}
Let $(\CF_t)_{t\geq0}$ denote the natural filtration of $Z$. By the
martingale convergence theorem, $(Z_t)$ converges almost surely to some
limit $Z_\infty$.
Fix numbers $T,S>0$ with $T>S+1$.
For $\ve>0$, define the bounded\vadjust{\goodbreak} stopping time $\tau_{\ve}:=\inf\{
t\geq0\dvtx  Z_t\leq\ve\}\wedge(S+1)$.
By the optional sampling theorem for right continuous supermartingales
(see, e.g., \cite{EthierKurtz1986}, Theorem II.2.13), we get
\[
Z_{\tau_\ve}\geq\E[Z_T | \CF_{\tau_\ve}]\qquad\mbox{a.s.}
\]
By Markov's inequality, we infer for $\delta>0$, that $\P[Z_T\geq
\delta | \CF_{\tau_{\ve}}]\leq Z_{\tau_{\ve}}/\delta$
and, in particular,
\[
\P[Z_T\geq\delta\mbox{ and }\tau_{\ve}\leq S ]\leq\frac
{\ve}{\delta}.
\]
Letting $\ve\downarrow0$ and then $\delta\downarrow0$ yields
\[
\P\bigl[Z_{T}>0\mbox{ and }\inf\{Z_{t}\dvtx  t\in[0,S]\}=0 \bigr]=0.
\]
Letting $T\to\infty$ gives
\[
\P\bigl[Z_{\infty}>0\mbox{ and }\inf\{Z_{t}\dvtx  t\in[0,S]\}=0 \bigr]=0.
\]
Since $(Z_t)$ converges, this implies
\[
\P[Z_{\infty}>0\mbox{ and }\inf\{Z_{t}\dvtx  t\geq0\}=0 ]=0.
\]
\upqed
\end{pf}
Recall $\delta$ defined in (\ref{Estar}) and define
\[
K:=\frac{2}{\delta} (M_{1,0}+M_{2,0} ).
\]
Let $X^K$ denote the truncated process defined in Section~\ref{S2.2}
with \mbox{$X^K_0=X_0$}. Recall that $(M^K_{i,t})_{t\geq0}$ is the total mass
process of $X^K_i$ and that it is a~martingale. Hence, we have
%
%
\begin{equation}
\label{EestF}
\P[F]\geq1-\delta,
\end{equation}
where
%
%
\begin{equation}
\label{EdefF}
F:= \{M^K_{1,t}+M^K_{2,t}\leq K/2\mbox{ for all }t\geq0 \}.
\end{equation}
Also
\[
\P[M_{1,t}+M_{2,t}\leq K/2\mbox{ for all }t\geq0 ] \geq
1-\delta.
\]
We can couple $X^K$ and $X$ such that both processes coincide on $F$. Hence,
for
\[
B:= \{M^K_{1,t}M^K_{2,t}\geq\delta\mbox{ for all }t\geq0 \},
\]
we have
%
%
\begin{equation}
\label{EE3.1}\qquad
\P[B] \geq\P[F\cap B] \geq\P[M_{1,t} M_{2,t}\geq
\delta\mbox{ for all }t\geq0 ]-\P[F^c] \geq5\delta.
\end{equation}

\subsection{Step 2: The lower bound for the conditional quadratic variation}
\label{S3.2}
Denote by $(\VV_t)_{t\geq0}$ the conditional quadratic variation
process of $(M^K_{1,t})_{t\geq0}$. Since $M^K_1$ is a martingale whose
jumps are bounded (by $K$), convergence of~$M^K_1$ implies almost sure
convergence of $\VV_t\to\VV_\infty<\infty$ as $t\to\infty$.
Hence, for any $\ve>0$ and $\delta>0$, there exists a $T_0$ such that
%
%
\begin{equation}
\label{EE3.3}
\P[\VV_t-\VV_{T_0}>\ve] \leq\delta\qquad\mbox{for
all }t\geq T_0.\vadjust{\goodbreak}
\end{equation}
The aim of this section is to show that (\ref{EE3.1}) leads to a
contradiction to (\ref{EE3.3}) which shows that the assumption that for
$X$ coexistence of types would be possible was wrong.

Fix $\delta>0$. We choose an appropriate $\ve>0$ for a contradiction
via the following procedure. Recall (\ref{E1.30}) and define
%
%
\begin{equation}
\label{Edefaast}
\TMstar(k):= 1- \TM(k,k)\quad\mbox{and}\quad
\TMstar:=\inf_{k\in S} \TMstar(k)>0.
\end{equation}
Recall $c$ from (\ref{E1.10.1}) and define
%
%
\begin{equation}
\label{equt:3}
c':= \frac{\TMstar c}{16} .
\end{equation}
Let
%
%
\begin{equation}
\label{equt:13}
R:=2+\frac{2K^2}{\delta^3 c'}.
\end{equation}
Furthermore, define
%
%
\begin{equation}
\label{equt:16}
\ve:=\frac{\delta^2c'}{R+1}.
\end{equation}
Note that
%
%
\begin{equation}
\label{equt:16b}
\ve-c'\delta^2= -R\ve\quad\mbox{and}\quad\frac{K^2}{R^2\ve}\leq
\delta.
\end{equation}

Recall from (\ref{EdefMK}) that $M^K_{1}(k)$, $k\in S$, are orthogonal
martingales and that $M^K_1=\sum_{k\in S}M^K_1(k)$. Hence, by (\ref
{Elowbousq}), for $t\geq T_0$, we have
\begin{eqnarray*}
\VV_t-\VV_{T_0} &=& \sum_{k\in S} \bigl( \langle M^K_{1}(k)
\rangle_t - \langle M^K_{1}(k) \rangle_{T_0} \bigr)\\
&\geq& \frac12\sum_{k\in S}\int_{T_0\wedge\tau^K}^{t\wedge
\tau^K}
\bigl(\TM X^K_{1,s}(k)X^K_{2,s}(k)+\TM X^K_{2,s}(k)X^K_{1,s}(k) \bigr)\,ds,
\end{eqnarray*}
where $\tau^K$ is defined in (\ref{equt:17_04_1}).
Define
%
%
\begin{equation}
Z_t=\frac12\sum_{k\in S} \int_{T_0}^{t}X^K_{1,s}(k)\CA X^K_{2,s}(k)+
X^K_{2,s}(k)\CA X^K_{1,s}(k)\,ds\qquad\mbox{for }t\geq T_0 .\hspace*{-35pt}
\end{equation}
Note that $\tau^K=\infty$ on $F$.
Hence, for all $ t\geq T_0$,
%
%
\begin{equation}
\label{EcompVZ}
\VV_t-\VV_{T_0}\geq Z_t\qquad\mbox{on } F.
\end{equation}

\begin{lemma}
\label{lem:2}
For any $k_1,k_2\in S$, $k_1\neq k_2$, and $t\geq T_0$, let
$N_t(k_1,k_2)$ be given by
\begin{eqnarray*}
&&N_t(k_1,k_2)\\
&& \qquad  = \mathop{\mathop{\sum}_{l,l'\in S}}_{ l\neq l'}\int
_{T_0}^t \SG
_{t-s}(k_1,l')\SG_{t-s}(k_2,l) \bigl(X^K_{1,s}(l')\,dM^K_{2,s}(l)
+X^K_{2,s}(l)\,dM^K_{1,s}(l') \bigr).
\end{eqnarray*}
Then for $t\geq T_0,$ on $F$ we have
\begin{eqnarray*}
&&X^K_{1,t}(k_1)X^K_{2,t}(k_2) \\[-1pt]
&& \qquad= \SG_{t-T_0}X^K_{1,T_0}(k_1) \SG
_{t-T_0}X^K_{2,T_0}(k_2)\\[-1pt]
&& \qquad \quad{} -\sum_{l\in S}\int_{T_0}^t \SG_{t-s}(k_1,l)\SG
_{t-s}(k_2,l)
\bigl(
X^K_{1,s}(l)\CA X^K_{2,s}(l) + X^K_{2,s}(l)\CA X^K_{1,s}(l) \bigr)\,
ds\\[-1pt]
&& \qquad \quad{} + N_t(k_1 ,k_2).
\end{eqnarray*}
\end{lemma}
\begin{pf}
This is an immediate consequence of (\ref{EdefMK}),
the fact that
\[
X^K_{1,t}(k)X^K_{2,t}(k)=0\qquad\mbox{for all }k\in S, \mbox{ if }
t\leq\tau^K,
\]
and of $\tau^K=\infty$ on $F$.
\end{pf}

\begin{lemma}
On the event $F$, the following decomposition holds:
%
%
\begin{equation}
\label{EdefZ}
Z_t=Z^{(1)}_t-Z^{(2)}_t+Z^{(3)}_{1,t}+Z^{(3)}_{2,t}\qquad\mbox{for
}t\geq T_0,
\end{equation}
where
\begin{eqnarray*}
Z^{(1)}_t&=&\sum_{l\in S} \int_{T_0}^t\SG_{s-T_0}X^K_{1,T_0}(l)
\bigl(\bar\CA(\SG_{s-T_0}X^K_{2,T_0})(l)
+\TMstar(l)\SG_{s-T_0}X^K_{2,T_0}(l) \bigr)\,ds,\\[-2pt]
Z^{(2)}_t&=&
\sum_{l\in S}\sum_{k\in S} \int_{T_0}^t\int_{T_0}^s \SG
_{s-r}(l,k)\bigl [\bigl(\bar\CA\SG_{s-r}(\cdot,k)\bigr)(l)
+\TMstar(l)\SG_{s-r}(l,k) \bigr]\\[-2pt]
&&\hphantom{\sum_{l\in S}\sum_{k\in S} \int_{T_0}^t\int_{T_0}^s}{}\times \bigl(X^K_{1,r}(k)\TM X^K_{2,r}(k)+\TM
X^K_{1,r}(k)X^K_{2,r}(k) \bigr)\,dr\,ds,\\[-2pt]
Z^{(3)}_{i,t}&=&\sum_{l\not= l'}\bar\TM(l,l')\\[-2pt]
&&
 {}\times
\sum_{k\in S}\int_{T_0}^t
\int_{T_0}^s
\SG_{s-r}(l,k)\\[-2pt]
&&\hphantom{{}\times\sum_{k\in S}\int_{T_0}^t
\int_{T_0}^s}
 {}\times \bigl(\SG_{s-r}X^K_{3-i,r}(l')-\SG
_{s-r}(l',k)X^K_{3-i,r}(k) \bigr)\,dM^K_{i,r}(k)\,ds.
\end{eqnarray*}
\end{lemma}
\begin{pf}
By definition,
\begin{eqnarray*}
Z_t&=&\frac12\sum_{i=1}^2\sum_{k\in S} \int
_{T_0}^{t}X^K_{i,s}(k)\CA
X^K_{3-i,s}(k)\,ds\\[-2pt]
&=&\frac12\sum_{i=1}^2\sum_{k\in S}
\int_{T_0}^{t}X^K_{i,s}(k)\sum_{l\in S}\bigl(\TM(k,l)-\1_{\{k=l\}}\bigr)
X^K_{3-i,s}(l)\,ds\\[-2pt]
&=&\mathop{\mathop{\sum}_{ k_1,k_2\in S}}_{ k_1\not=k_2}
\int_{T_0}^{t}\bar\TM(k_1,k_2)X^K_{1,s}(k_1)X^K_{2,s}(k_2)\,ds,
\end{eqnarray*}
where the last inequality follows since $X^K_{1,t}(k)X^K_{2,t}(k)=0$
for all $k\in S,
t\geq0$ on~$F$.
Now, by Lemma~\ref{lem:2}, we have
\begin{eqnarray*}
Z_t&=&\mathop{\mathop{\sum}_{ k_1,k_2\in S}}_{ k_1\not=k_2}
\int_{T_0}^{t}\bar\TM(k_1,k_2) \SG_{s-T_0}X^K_{1,T_0}(k_1) \SG
_{s-T_0}X^K_{2,T_0}(k_2)\,ds\\[-2pt]
&&{}-\mathop{\mathop{\sum}_{ k_1,k_2\in S}}_{ k_1\not=k_2} \int
_{T_0}^{t}\bar\TM
(k_1,k_2) \sum_{l\in S}\int_{T_0}^s \SG_{s-r}(k_1,l)\SG
_{s-r}(k_2,l) \\[-2pt]
&&\hphantom{-\mathop{\mathop{\sum}_{ k_1,k_2\in S}}_{ k_1\not=k_2} \int
_{T_0}^{t}\bar\TM
(k_1,k_2) \sum_{l\in S}\int_{T_0}^s}
{} \times\bigl(
X^K_{1,r}(l)\CA X^K_{2,r}(l) + X^K_{2,r}(l)\CA X^K_{1,r}(l) \bigr)\,
dr\,ds
\\[-2pt]
&&
{}+\mathop{\mathop{\sum}_{ k_1,k_2\in S}}_{ k_1\not=k_2}\int
_{T_0}^{t}\bar\TM
(k_1,k_2) N_s(k_1,k_2)\,ds.
\end{eqnarray*}
Let us consider the first term on the right-hand side of the above equation.
We easily get that
\begin{eqnarray*}
&&\mathop{\mathop{\sum}_{ k_1,k_2\in S}}_{ k_1\not=k_2}
\int_{T_0}^{t}\bar\TM(k_1,k_2) \SG_{s-T_0}X^K_{1,T_0}(k_1) \SG
_{s-T_0}X^K_{2,T_0}(k_2)\,ds
\\[-2pt]
&& \qquad = \sum_{ k_1\in S}
\int_{T_0}^{t}\SG_{s-T_0}X^K_{1,T_0}(k_1)
\bigl(\bar\CA\SG_{s-T_0}X^K_{2,T_0}(k_1) +\TMstar(k_1)\SG
_{s-T_0}X^K_{2,T_0}(k_1)
\bigr)\,ds\\[-2pt]
&& \qquad = Z^{(1)}_t
\end{eqnarray*}
and we are done with the first term in the decomposition.
Similarly, for the second term, we have
\begin{eqnarray*}
&&\mathop{\mathop{\sum}_{ k_1,k_2\in S}}_{ k_1\not=k_2} \int
_{T_0}^{t}\bar\TM
(k_1,k_2) \sum_{l\in S}\int_{T_0}^s \SG_{s-r}(k_1,l)\SG
_{s-r}(k_2,l) \\[-2pt]
&&
\hphantom{\mathop{\mathop{\sum}_{ k_1,k_2\in S}}_{ k_1\not=k_2} \int
_{T_0}^{t}\bar\TM
(k_1,k_2) \sum_{l\in S}\int_{T_0}^s}
 {} \times\bigl(
X^K_{1,r}(l)\CA X^K_{2,r}(l) + X^K_{2,r}(l)\CA X^K_{1,r}(l) \bigr)\,
dr\,ds
\\[-2pt]
&& \qquad= \sum_{l\in S}\sum_{ k_1\in S}\int_{T_0}^{t}\int_{T_0}^s
\SG
_{s-r}(k_1,l) [\bar\CA\SG_{s-r}(\cdot,l)(k_1)+\TMstar
(k_1)\SG_{s-r}(k_1,l) ]\\[-2pt]
\nonumber
&& \qquad \quad
\hphantom{\sum_{l\in S}\sum_{ k_1\in S}\int_{T_0}^{t}\int_{T_0}^s}
{} \times\bigl(
X^K_{1,r}(l)\CA X^K_{2,r}(l) + X^K_{2,r}(l)\CA X^K_{1,r}(l) \bigr)\,
dr\,ds\\[-2pt]
&& \qquad= Z^{(2)}_t .
\end{eqnarray*}

For the third term, we have
\begin{eqnarray*}
&&\hspace*{-3pt}\mathop{\mathop{\sum}_{ k_1,k_2\in S}}_{ k_1\not
=k_2}\int_{T_0}^{t}\bar\TM
(k_1,k_2) N_s(k_1,k_2)\,ds\\[-2pt]
&&\hspace*{-3pt} \qquad= \mathop{\mathop{\sum}_{ k_1,k_2\in S}}_{
k_1\not=k_2}\int_{T_0}^{t}\bar\TM(k_1,k_2)
\mathop{\mathop{\sum}_{l,l'\in S}}_{ l\neq l'}\int_{T_0}^s \SG
_{s-r}(k_1,l')\SG
_{s-r}(k_2,l)\\[-2pt] 
&&\hspace*{-3pt} \qquad \quad\hphantom{\mathop{\mathop{\sum}_{
k_1,k_2\in S}}_{ k_1\not=k_2}\int_{T_0}^{t}\bar\TM(k_1,k_2)
\mathop{\mathop{\sum}_{l,l'\in S}}_{ l\neq l'}\int_{T_0}^s}
{} \times\bigl(X^K_{1,r}(l')\,dM^K_{2,r}(l)
+X^K_{2,r}(l)\,dM^K_{1,r}(l') \bigr)\,ds\\[-2pt]
&&\hspace*{-3pt} \qquad = \mathop{\mathop{\sum}_{ k_1,k_2\in S}}_{
k_1\not=k_2}\bar\TM(k_1,k_2) \\[-2pt]
&&\hspace*{-3pt} \qquad \quad\hphantom{\mathop{\mathop{\sum}_{
k_1,k_2\in S}}_{ k_1\not=k_2}}
{}\times\int
_{T_0}^{t} \int_{T_0}^s \biggl\{ \sum_{l\in S} \SG
_{s-r}(k_2,l)\\[-2pt]
&&\hspace*{-3pt} \qquad \quad
\hphantom{\mathop{\mathop{\sum}_{
k_1,k_2\in S}}_{ k_1\not=k_2}
{}\times\int
_{T_0}^{t} \int_{T_0}^s \biggl\{ \sum_{l\in S}}
{}\times\bigl(\SG
_{s-r}X^K_{1,r}(k_1) - \SG_{s-r}(k_1,l)X^K_{1,r}(l) \bigr)\,d M^K_{2,r}(l)
\\[-2pt]
&&\hspace*{-3pt} \qquad \quad
\hphantom{\mathop{\mathop{\sum}_{ k_1,k_2\in S}}_{ k_1\not=k_2}
{}\times\int
_{T_0}^{t} \int_{T_0}^s \biggl\{}
{}
+ \sum_{l'\in S} \SG_{s-r}(k_1,l') \bigl(\SG
_{s-r}X^K_{2,r}(k_2)\\[-2pt]
&&\hspace*{-3pt} \qquad \quad
\hphantom{\mathop{\mathop{\sum}_{ k_1,k_2\in S}}_{ k_1\not=k_2}
{}\times\int
_{T_0}^{t} \int_{T_0}^s \biggl\{
{}
+ \sum_{l'\in S} \SG_{s-r}(k_1,l') \bigl(}
{}- \SG_{s-r}(k_2,l')X^K_{2,r}(l') \bigr)\,d M^K_{1,r}(l')
\biggr\}\,ds\\[-2pt]
&&\hspace*{-3pt} \qquad = \mathop{\mathop{\sum}_{ k_1,k_2\in S}}_{
k_1\not=k_2} \bar\TM(k_1,k_2)\\[-2pt]
&&\hspace*{-3pt} \qquad \quad
\hphantom{\mathop{\mathop{\sum}_{ k_1,k_2\in S}}_{ k_1\not=k_2}}
{}\times\int
_{T_0}^{t} \int_{T_0}^s \biggl\{ \sum_{l\in S} \SG_{s-r}(k_2,l)
\\[-2pt]
&&\hspace*{-3pt} \qquad \quad
\hphantom{\mathop{\mathop{\sum}_{ k_1,k_2\in S}}_{ k_1\not=k_2}
{}\times\int
_{T_0}^{t} \int_{T_0}^s \biggl\{ \sum_{l\in S}}
{}\times\bigl(\SG
_{s-r}X^K_{1,r}(k_1)- \SG_{s-r}(k_1,l)X^K_{1,r}(l) \bigr)\,d M^K_{2,r}(l)
\\[-2pt]
&&\hspace*{-3pt} \qquad \quad
\hphantom{\mathop{\mathop{\sum}_{ k_1,k_2\in S}}_{ k_1\not=k_2}
{}\times\int
_{T_0}^{t} \int_{T_0}^s \biggl\{}
{}
+ \sum_{l\in S} \SG_{s-r}(k_2,l) \bigl(\SG
_{s-r}X^K_{2,r}(k_1)\\[-2pt]
&&\hspace*{-3pt} \qquad \quad
\hphantom{\mathop{\mathop{\sum}_{ k_1,k_2\in S}}_{ k_1\not=k_2}
{}\times\int
_{T_0}^{t} \int_{T_0}^s \biggl\{
{}
+ \sum_{l\in S} \SG_{s-r}(k_2,l) \bigl(}\hspace*{8pt}
{}- \SG_{s-r}(k_1,l)X^K_{2,r}(l) \bigr)\,d M^K_{1,r}(l)
\biggr\}\,ds\\[-2pt]
&&\hspace*{-3pt} \qquad=Z^{(3)}_{1,t}+Z^{(3)}_{2,t},
\end{eqnarray*}
where the third equality follows by the symmetry of $\bar\TM$.
\end{pf}
\begin{lemma}\label{LboundZ2}
Recall $\bar G^*$ from (\ref{E1.9_1}). For all $t\geq T_0$, on $F$ we have
\[
Z^{(2)}_t\leq\frac{\bar G^*_{2(t-T_0)}}{2}Z_t.
\]
\end{lemma}
\begin{pf}
Recall (\ref{E5.2}) and (\ref{E5.4_1}) and note that (with $I$ the
unit matrix)
%
%
\begin{equation}
\label{Epik}
\int_0^t \SG^T_r\bar\CA\SG_r\,dr=\frac12 (\bar\SG_{2t}-I).
\end{equation}
Hence [using that $\TMstar(l)\leq1$],
\begin{eqnarray*}
Z^{(2)}_t&\leq&
\sum_{k\in S} \int_{T_0}^t\int_{r}^t\bigl [ \SG^T_{s-r}\bar\CA\SG_{s-r}
+\bar\SG_{2(s-r)} \bigr](k,k)\,ds\\
&&\hphantom{\sum_{k\in S} \int_{T_0}^t\int_{r}^t}
{} \times
[X^K_{1,r}(k)\TM X^K_{2,r}(k)+\TM X^K_{1,r}(k)X^K_{2,r}(k) ]\,dr\\
&=&\sum_{k\in S} \int_{T_0}^t \biggl( \frac{1}{2} \bigl(\bar\SG
_{2(t-r)}(k,k)-1\bigr) + \int_{r}^{t} \bar\SG_{2(s-r)}(k,k) \,ds \biggr)\\
&&\hphantom{\sum_{k\in S} \int_{T_0}^t}
{} \times[X^K_{1,r}(k)\TM X^K_{2,r}(k)+\TM
X^K_{1,r}(k)X^K_{2,r}(k) ]\,dr.
\end{eqnarray*}
Then use the fact that $\bar\SG_t(k,k)-1\leq0$
and
\[
\TM X^K_{i,r}(k)X^K_{3-i,r}(k)= \CA X^K_{i,r}(k)X^K_{3-i,r}(k) \qquad
\mbox{on }F
\]
to get
that on $F$, $Z_t^{(2)}$ is bounded above by
\begin{eqnarray*}
&&\frac12\sum_{k\in S} \int_{T_0}^t \bar G_{2(t-r)}(k,k)
\bigl(\CA X^K_{1,r}(k)X^K_{2,r}(k)+X^K_{1,r}(k)\CA X^K_{2,r}(k)
\bigr)\,dr
\nonumber\\
&& \qquad\leq\frac12 \bar G^{*}_{2(t-T_0)}\int_{T_0}^t
\sum_{k\in S} \bigl(\CA X^K_{1,r}(k)X^K_{2,r}(k)+X^K_{1,r}(k)\CA
X^K_{2,r}(k) \bigr)\,dr\nonumber\\
&& \qquad=\frac12 \bar G^{*}_{2(t-T_0)}Z_t.
\end{eqnarray*}
\upqed
\end{pf}
Next we will handle $Z^{(3)}_{i}$.
\begin{lemma}
\label{Lemma3.6}
For all $t\geq0$ and for $i=1,2$, we have
%
%
\begin{equation}
\label{equt:11}
Z^{(3)}_{i,t}= \sum_{k\in S}\int_{T_0}^t h_i(k,t,r)\,dM^K_{i,r}(k),
\end{equation}
where for $t$ sufficiently large
\[
|h_i(k,t,r)|\leq \bar G^*_{2(t-T_0)}M^K_{3-i,r}\qquad\mbox{for all
}r\in[T_0,t].
\]
\end{lemma}
\begin{pf}
First, by the stochastic Fubini theorem (see, e.g., Theorem~IV.64 in \cite
{Protter2004}),
we can change the order of integration in order to get
\begin{eqnarray*}
h_i(k,t,r)&=&\mathop{\mathop{\sum}_{l,l'\in S}}_{ l\neq l'}\bar\TM
(l,l')\int_r^t
\SG_{s-r}(l,k) \bigl(\SG_{s-r}X^K_{3-i,r}(l')-\SG
_{s-r}(l',k)X^K_{3-i,r}(k) \bigr)\,ds\\
&=&\int_r^t \SG^T_{s-r}\bar\CA\SG_{s-r} X^K_{3-i,r}(k)-\SG
^T_{s-r}\bar\CA\SG_{s-r}(k,k)X^K_{3-i,r}(k)\,ds\\
&&{} +\int_r^t\sum_{l\in S}\TMstar(l)\SG_{s-r}(l,k)[\SG
_{s-r}X^K_{3-i,r}(l)-\SG_{s-r}(l,k)X^K_{3-i,r}(k)]\,ds\\
&=:&I_1+I_2.
\end{eqnarray*}
For the first integral, using (\ref{Epik}), we get
\[
0 \leq\tfrac12\bar\SG_{2(t-r)}X^K_{3-i,r}(k)-\tfrac12\bar\SG
_{2(t-r)}(k,k)X^K_{3-i,r}(k) = I_1 \leq \tfrac12 M^K_{3-i,r}.
\]
For the second integral, we obtain similarly
\begin{eqnarray*}
0&\leq& I_2\leq\int_r^t\sum_{l\in S} \SG_{s-r}(l,k)\SG
_{s-r}X^K_{3-i,r}(l)\,ds\\
&=&\int_r^t\bar\SG_{s-r}X^K_{3-i,r}(k)\,ds\\
&=& \frac12\bar G^*_{2(t-r)}M^K_{3-i,r} \leq \frac12\bar
G^*_{2(T_0-r)}M^K_{3-i,r}.
\end{eqnarray*}
Combining the estimates for $I_1$ and $I_2$, we get
\[
|h_i(k,t,r)| \leq \tfrac12 \bigl(\bar G^*_{2(t-T_0)}+1 \bigr)
M^K_{3-i,r} \leq \bar G^*_{2(t-T_0)} M^K_{3-i,r},
\]
and we get the bound for $h_i(k,t,r)$ for $t$ sufficiently large.
\end{pf}
We have to introduce some notation and define a number of additional constants.
Define the event
%
%
\begin{equation}
\label{EdefC}
C=\biggl \{ \sum_{|k|\leq L} X^K_{i,T_0}(k) \geq \frac
{1}{2}M^K_{i,T_0} \mbox{ for }i=1,2 \biggr\}.
\end{equation}
Since $M^K_{i,T_0}<\infty$, $i=1,2$, there exists an $L>0$ such that
%
%
\begin{equation}
\label{EestC}\P[C]\geq1-\delta.
\end{equation}

\begin{lemma}
\label{Lemma3.7}
There exists a $ T_1\geq T_0$ such that for all $t>T_1$,
\[
Z^{(1)}_t\geq
c' \bar G^*_{2(t-T_0)} M^K_{1,T_0}M^K_{2,T_0}
\qquad\mbox{on }C\cap F.
\]
\end{lemma}
\begin{pf}
In order to simplify the notation, let
$Z^{(1)}_t=Z^{(1,1)}_t+Z^{(1,2)}_t$, where
\begin{eqnarray*}
Z^{(1,1)}_t&:=& \sum_{l\in S}\int_{T_0}^{t}\SG
_{s-T_0}X^K_{1,T_0}(l)\bar\CA(\SG_{s-T_0}X^K_{2,T_0})(l)\,ds,\\
Z^{(1,2)}_t&:=&
\sum_{l\in S}\int_{T_0}^{t}\SG_{s-T_0}X^K_{1,T_0}(l) \TMstar(l)\SG
_{s-T_0}X^K_{2,T_0}(l)\,ds.
\end{eqnarray*}
We start with showing that $Z^{(1,1)}_t\geq0$. To this end, using
(\ref{Epik}), we compute
\begin{eqnarray*}
Z^{(1,1)}_t&=&\sum_{l\in S}\int_{T_0}^{t}X^K_{1,T_0}(l) [\SG
^T_{s-T_0}\bar\CA\SG_{s-T_0}]X^K_{2,T_0}(l)\,ds\\[-2pt]
&=&\frac12\sum_{l\in S}\bar X^K_{1,T_0}(l)\bigl(\bar\SG
_{2(t-T_0)}-I\bigr)X^K_{2,T_0}(l)\\[-2pt]
&=&\frac12\sum_{l\in S}\bar X^K_{1,T_0}(l)\bar\SG
_{2(t-T_0)}X^K_{2,T_0}(l) \geq0\qquad\mbox{on }F,
\end{eqnarray*}
since $X^K_{1,T_0}(k)X^K_{2,T_0}(k)=0$ on $F$ for all $k\in S$.

The bound for $Z^{(1,2)}$ follows similarly to Dawson and Perkins \cite
{DawsonPerkins1998}, page~1109. Recall $c$ from (\ref{E1.10.1}). For
$k,l\in S$ such that $|k|,|l|\leq L$, let $T(k,l)$ be large enough such that
\[
\frac{\bar G_t(k,l)}{\bar G_t^*}\geq\frac{c}{2}\qquad\mbox{for all
}t\geq T(k,l).
\]
Define
%
%
\begin{equation}
\label{equt:6}
T_1=T_0+\max_{|k|,|l|\leq L} T(k,l) .
\end{equation}
Then for any $t\geq T_1$, we have that on $C$
\begin{eqnarray*}
Z^{(1,2)}_t&\geq& \frac{\TMstar}{2}\sum_{k,l\in S} \bar
G_{2(t-T_0)}(k,l)X^K_{1,T_0}(k)X^K_{2,T_0}(l)\\[-2pt]
&\geq& \frac{\TMstar}{2} \biggl(\sum_{|k|,|l|\leq L}
X^K_{1,T_0}(k)X^K_{2,T_0}(l) \biggr)
\min_{|k|, |l|\leq L}\bar G_{2(t-T_0)}(k,l)\\[-2pt]
&\geq&\frac{\TMstar c}{16} M^K_{1,T_0}M^K_{2,T_0} \bar
G^*_{2(t-T_0)} ,
\end{eqnarray*}
where the last inequality follows by (\ref{EdefC}) and (\ref{equt:6}).
Recalling (\ref{equt:3}), we get that on $C$
\[
Z^{(1,2)}_t
\geq
c' \bar G^*_{2(t-T_0)} M^K_{1,T_0}M^K_{2,T_0}.
\]
Since
$Z^{(1,1)}_T\geq0$ on $F$, this finishes the proof of Lemma~\ref{Lemma3.7}.
\end{pf}

From Lemmas~\ref{LboundZ2} and~\ref{Lemma3.7}, we get that on
$F\cap C$, $Z_t$ is bounded below by
%
%
\begin{equation}
\label{equt:4}
Z_t \geq c' \bar G^*_{2(t-T_0)} M^K_{1,T_0}M^K_{2,T_0}
- \frac{ \bar G^*_{2(t-T_0)}}{2}Z_t + \bar Z^{(3)}_t ,
\end{equation}
where $\bar Z^{(3)}_t=Z^{(3)}_{1,t}+Z^{(3)}_{2,t}.$
Let $\alpha:=1/(1+\bar G^*_{2(t-T_0)}/2)$.
From (\ref{equt:4}), we get that on $F\cap C$
%
%
\begin{equation}
Z_{t} \geq\alpha c' \bar G^*_{2(t-T_0)} M^K_{1,T_0}M^K_{2,T_0}+
\alpha\bar Z^{(3)}_t.\vadjust{\goodbreak}
\end{equation}
Then [recall from (\ref{EcompVZ}) that $Z_t\leq\VV_t-\VV_{T_0}$ on $F$]
%
%
\begin{eqnarray}
\label{EfE}
&&\P[\VV_t-\VV_{T_0} \leq\ve, F\cap C ]\nonumber\\[-1pt]
&&  \qquad=\P[Z_t\leq\ve, \VV_t-\VV_{T_0}\leq\ve, F\cap C ]\nonumber
\\[-9pt]
\\[-9pt]
&&  \qquad\leq\P\bigl[
\alpha\bar Z^{(3)}_t\leq\ve-\alpha c' \bar G^*_{2(t-T_0)}
M^K_{1,T_0}M^K_{2,T_0},\nonumber\\[-1pt]
&& \hspace*{32pt}\qquad\hphantom{\leq\P\bigl[} \VV_t-\VV_{T_0}\leq\ve, F\cap C \bigr].\nonumber
\end{eqnarray}

We assume that $t$ is sufficiently large so that $\bar
G^*_{2(t-T_0)}\geq2$, hence
%
%
\begin{equation}
\label{Econdi035}
1\leq\alpha\bar G^*_{2(t-T_0)}=\frac{2 \bar G^*_{2(t-T_0)}}{2+
\bar G^*_{2(t-T_0)}}\leq2 .
\end{equation}
By (\ref{Econdi035}), (\ref{EE3.1}) and (\ref{equt:16b}), we get
\begin{eqnarray*}
&&\P\bigl[
\alpha\bar Z^{(3)}_t\leq\ve-\alpha c' \bar G^*_{2(t-T_0)}
M^K_{1,T_0}M^K_{2,T_0},
\VV_t-\VV_{T_0}\leq\ve, F\cap C \bigr]\\[-1pt]
&& \qquad\leq\P\bigl[
\alpha\bar Z^{(3)}_t\leq\ve-\alpha c' \bar G^*_{2(t-T_0)}
\delta^2,
\VV_t-\VV_{T_0}\leq\ve, F\cap C \bigr]+1-5\delta
\\[-1pt]
&& \qquad\leq\P\bigl[
\alpha\bar Z^{(3)}_t\leq\ve-c'\delta^2,
\VV_t-\VV_{T_0}\leq\ve, F\cap C \bigr]+1-5\delta\\[-1pt]
&& \qquad= \P\bigl[
\alpha\bar Z^{(3)}_t\leq-R\ve,
\VV_t-\VV_{T_0}\leq\ve, F\cap C \bigr]+1-5\delta\\[-1pt]
&& \qquad\leq R^{-2}\ve^{-2}\alpha^2\E\bigl[\bigl(\bar Z^{(3)}_t\bigr)^2\1_{\{
{{\langle M^K_{1,\cdot}\rangle}}t-{{\langle M^K_{1,\cdot}\rangle
}}{T_0}\leq\ve\}} \1_F \bigr]+1-5\delta.
\end{eqnarray*}
Using Lemma~\ref{Lemma3.6}, this inequality can be continued by
%
%
\begin{eqnarray}
\label{equt:15}
&\leq&
R^{-2}\ve^{-1} (K/2)^2\bigl(\alpha\bar G^*_{2(t-T_0)}\bigr)^2 +1-5\delta
\nonumber\\[-1pt]
&\leq&
\frac{K^2}{R^2\ve} +1-5\delta
\\[-1pt]
&\leq&1-4\delta.\nonumber
\end{eqnarray}
Combining (\ref{EfE}), (\ref{equt:15}), (\ref{EestF}) and (\ref
{EestC}), we get
\begin{eqnarray*}
&&\P[\VV_t-\VV_{T_0}\leq\ve]\\[-1pt]
&& \qquad\leq\P[ \VV_t-\VV
_{T_0}\leq\ve, F\cap C ] +\P[F^c]+\P[C^c]\\[-1pt]
&& \qquad\leq1-4\delta+\delta+\delta=1-2\delta.
\end{eqnarray*}
This is a contradiction to (\ref{EE3.3}) and hence finishes the proof of
Theorem~\ref{T2}.

%

\printaddresses


\begin{thebibliography}{18}

\bibitem{CliffordSudbury1973}
%
\begin{barticle}[mr]
\bauthor{\bsnm{Clifford},~\bfnm{Peter}\binits{P.}} \AND
\bauthor{\bsnm{Sudbury}, \bfnm{Aidan}\binits{A.}}
(\byear{1973}).
\btitle{A model for spatial conflict}.
\bjournal{Biometrika}
\bvolume{60}
\bpages{581--588}.
\bid{mr={0343950}}
\end{barticle}
%
\endbibitem

\bibitem{CoxGreven1994a}
%
\begin{barticle}[mr]
\bauthor{\bsnm{Cox},~\bfnm{J.~T.}\binits{J.~T.}} \AND
\bauthor{\bsnm{Greven},~\bfnm{Andreas}\binits{A.}}
(\byear{1994}).
\btitle{Ergodic theorems for infinite systems of locally interacting
diffusions}.
\bjournal{Ann. Probab.}
\bvolume{22}
\bpages{833--853}.
\bid{mr={1288134}}
\end{barticle}
%
\endbibitem

\bibitem{CoxKlenkePerkins2000}
%
\begin{bincollection}[mr]
\bauthor{\bsnm{Cox},~\bfnm{J.~Theodore}\binits{J.~T.}},
\bauthor{\bsnm{Klenke},~\bfnm{Achim}\binits{A.}} \AND
\bauthor{\bsnm{Perkins},~\bfnm{Edwin~A.}\binits{E.~A.}}
(\byear{2000}).
\btitle{Convergence to equilibrium and linear systems duality}.
In \bbooktitle{Stochastic Models ({O}ttawa, {ON}, 1998)}
(\beditor{L.~B.~Gorostiza and B.~G. Ivanoff}, eds.).
\bseries{CMS Conf. Proc.}
\bvolume{26}
\bpages{41--66}.
\bpublisher{Amer. Math. Soc.}, \baddress{Providence, RI}.
\bid{mr={1765002}}
\end{bincollection}\vadjust{\goodbreak}
%
\endbibitem

\bibitem{DawsonFleischmann1985}
%
\begin{barticle}[mr]
\bauthor{\bsnm{Dawson},~\bfnm{Donald~A.}\binits{D.~A.}} \AND
\bauthor{\bsnm{Fleischmann},~\bfnm{Klaus}\binits{K.}}
(\byear{1985}).
\btitle{Critical dimension for a model of branching in a random medium}.
\bjournal{Z. Wahrsch. Verw. Gebiete}
\bvolume{70}
\bpages{315--334}.
\bid{doi={10.1007/BF00534864}, mr={0803673}}
\end{barticle}
%
\endbibitem

\bibitem{DawsonPerkins1998}
%
\begin{barticle}[mr]
\bauthor{\bsnm{Dawson},~\bfnm{Donald~A.}\binits{D.~A.}} \AND
\bauthor{\bsnm{Perkins},~\bfnm{Edwin~A.}\binits{E.~A.}}
(\byear{1998}).
\btitle{Long-time behavior and coexistence in a mutually catalytic branching
model}.
\bjournal{Ann. Probab.}
\bvolume{26}
\bpages{1088--1138}.
\bid{doi={10.1214/aop/1022855746}, mr={1634416}}
\end{barticle}
%
\endbibitem

\bibitem{EthierKurtz1986}
%
\begin{bbook}[mr]
\bauthor{\bsnm{Ethier},~\bfnm{Stewart~N.}\binits{S.~N.}} \AND
\bauthor{\bsnm{Kurtz},~\bfnm{Thomas~G.}\binits{T.~G.}}
(\byear{1986}).
\btitle{Markov Processes: Characterization and Convergence}.
\bpublisher{Wiley}, \baddress{New York}.
\bid{doi={10.1002/9780470316658}, mr={0838085}}
\end{bbook}
%
\endbibitem

\bibitem{HolleyLiggett1981}
%
\begin{barticle}[mr]
\bauthor{\bsnm{Holley},~\bfnm{Richard}\binits{R.}} \AND
\bauthor{\bsnm{Liggett},~\bfnm{Thomas~M.}\binits{T.~M.}}
(\byear{1981}).
\btitle{Generalized potlatch and smoothing processes}.
\bjournal{Z.~Wahrsch. Verw. Gebiete}
\bvolume{55}
\bpages{165--195}.
\bid{doi={10.1007/BF00535158}, mr={0608015}}
\end{barticle}
%
\endbibitem

\bibitem{HolleyLiggett1975}
%
\begin{barticle}[mr]
\bauthor{\bsnm{Holley},~\bfnm{Richard~A.}\binits{R.~A.}} \AND
\bauthor{\bsnm{Liggett},~\bfnm{Thomas~M.}\binits{T.~M.}}
(\byear{1975}).
\btitle{Ergodic theorems for weakly interacting infinite systems and
the voter
model}.
\bjournal{Ann. Probab.}
\bvolume{3}
\bpages{643--663}.
\bid{mr={0402985}}
\end{barticle}
%
\endbibitem

\bibitem{Kallenberg1977}
%
\begin{barticle}[mr]
\bauthor{\bsnm{Kallenberg},~\bfnm{Olav}\binits{O.}}
(\byear{1977}).
\btitle{Stability of critical cluster fields}.
\bjournal{Math. Nachr.}
\bvolume{77}
\bpages{7--43}.
\bid{mr={0443078}}
\end{barticle}
%
\endbibitem

\bibitem{KM2}
%
\begin{bmisc}[auto:STB|2010-11-18|09:18:59]
\bauthor{\bsnm{Klenke},~\bfnm{Achim}\binits{A.}} \AND
\bauthor{\bsnm{Mytnik},~\bfnm{Leonid}\binits{L.}}
(\byear{2008}).
\bhowpublished{Infinite rate mutually catalytic branching in
infinitely many
colonies. Construction, characterization and convergence. Preprint. Available
at \href{http://arxiv.org/abs/0901.0623}{arXiv:0901.0623} [math.PR]}.
\end{bmisc}
%
\endbibitem

\bibitem{KM1}
%
\begin{barticle}[mr]
\bauthor{\bsnm{Klenke},~\bfnm{Achim}\binits{A.}} \AND
\bauthor{\bsnm{Mytnik},~\bfnm{Leonid}\binits{L.}}
(\byear{2010}).
\btitle{Infinite rate mutually catalytic branching}.
\bjournal{Ann. Probab.}
\bvolume{38}
\bpages{1690--1716}.
\bid{doi={10.1214/09-AOP520}, mr={2663642}}
\end{barticle}
%
\endbibitem

\bibitem{KO}
%
\begin{barticle}[mr]
\bauthor{\bsnm{Klenke},~\bfnm{Achim}\binits{A.}} \AND
\bauthor{\bsnm{Oeler},~\bfnm{Mario}\binits{M.}}
(\byear{2010}).
\btitle{A {T}rotter-type approach to infinite rate mutually catalytic
branching}.
\bjournal{Ann. Probab.}
\bvolume{38}
\bpages{479--497}.
\bid{doi={10.1214/09-AOP488}, mr={2642883}}
\end{barticle}
%
\endbibitem

\bibitem{Liggett1985}
%
\begin{bbook}[mr]
\bauthor{\bsnm{Liggett},~\bfnm{Thomas~M.}\binits{T.~M.}}
(\byear{1985}).
\btitle{Interacting Particle Systems}.
\bseries{Grundlehren der Mathematischen Wissenschaften [Fundamental Principles
of Mathematical Sciences]}
\bvolume{276}.
\bpublisher{Springer}, \baddress{New York}.
\bid{mr={0776231}}
\end{bbook}
%
\endbibitem

\bibitem{NotoharaShiga1980}
%
\begin{barticle}[mr]
\bauthor{\bsnm{Notohara},~\bfnm{Morihiro}\binits{M.}} \AND
\bauthor{\bsnm{Shiga},~\bfnm{Tokuzo}\binits{T.}}
(\byear{1980}).
\btitle{Convergence to genetically uniform state in stepping stone
models of
population genetics}.
\bjournal{J. Math. Biol.}
\bvolume{10}
\bpages{281--294}.
\bid{doi={10.1007/BF00276987}, mr={0599812}}
\end{barticle}
%
\endbibitem

\bibitem{Oeler2008}
%
\begin{bmisc}[auto:STB|2010-11-18|09:18:59]
\bauthor{\bsnm{Oeler},~\bfnm{Mario}\binits{M.}}
(\byear{2008}).
\bhowpublished{Mutually catalytic branching at infinite rate. Ph.D. thesis,
Univ. Mainz}.
\end{bmisc}
%
\endbibitem

\bibitem{Protter2004}
%
\begin{bbook}[mr]
\bauthor{\bsnm{Protter},~\bfnm{Philip~E.}\binits{P.~E.}}
(\byear{2004}).
\btitle{Stochastic Integration and Differential Equations},
\bedition{2nd} ed.
\bseries{Applications of Mathematics (New York)}
\bvolume{21}.
\bpublisher{Springer}, \baddress{Berlin}.
\bid{mr={2020294}}
\end{bbook}
%
\endbibitem

\bibitem{Shiga1980a}
%
\begin{barticle}[mr]
\bauthor{\bsnm{Shiga},~\bfnm{Tokuzo}\binits{T.}}
(\byear{1980}).
\btitle{An interacting system in population genetics}.
\bjournal{J. Math. Kyoto Univ.}
\bvolume{20}
\bpages{213--242}.
\bid{mr={0582165}}
\end{barticle}
%
\endbibitem

\end{thebibliography}
\end{document}